\documentclass[reqno,11pt]{preprint}

\usepackage[full]{textcomp}
\usepackage{times}

\usepackage[cal=boondoxo]{mathalfa}
\usepackage{microtype}

\usepackage{amsmath}
\usepackage{amsfonts}
\usepackage{amssymb}
\usepackage{stmaryrd}

\usepackage[all]{xy}
\usepackage[utf8]{inputenc} 
\usepackage{comment}

\usepackage{hyperref}
\usepackage{breakurl}

\usepackage{mathrsfs}
\usepackage{enumitem}

\usepackage{tikz}
\usepackage{dsfont}
\usepackage{mhequ} 
\usepackage{mhsymb} 
\usepackage{mhenvs} 
\usepackage{array}

\usepackage{wasysym}
\usepackage{centernot}
\usepackage{mathtools}

\setlist{noitemsep,nolistsep,leftmargin=1.7em}
\usetikzlibrary{snakes}
\usetikzlibrary{decorations}
\usetikzlibrary{positioning}
\usetikzlibrary{shapes}
\usetikzlibrary{external}

\DeclareFontFamily{U}{mathx}{\hyphenchar\font45}
\DeclareFontShape{U}{mathx}{m}{n}{
      <5> <6> <7> <8> <9> <10>
      <10.95> <12> <14.4> <17.28> <20.74> <24.88>
      mathx10
      }{}
\DeclareSymbolFont{mathx}{U}{mathx}{m}{n}
\DeclareMathSymbol{\bigtimes}{1}{mathx}{"91}

\def\emptyset{{\centernot\ocircle}}

\definecolor{darkred}{rgb}{0.7,0.1,0.1}
\definecolor{darkblue}{rgb}{0.1,0.1,0.8}
\definecolor{darkgreen}{rgb}{0.1,0.7,0.1}
\providecommand{\figures}{false}
{ \ifthenelse{\equal{\figures}{false}} {#1}{\[ {\rm Figure \ missing !} \]} }{}

\def\CP{\mathcal{P}}

\def\CM{\mathcal{M}}

\def\${|\!|\!|}

\makeatletter
\newcommand*\bigcdot{\mathpalette\bigcdot@{.5}}
\newcommand*\bigcdot@[2]{\mathbin{\vcenter{\hbox{\scalebox{#2}{$\m@th#1\bullet$}}}}}
\makeatother

\def\CF{\mathcal{F}}

\newenvironment{DIFnomarkup}{}{} 

\newtheorem{example}[lemma]{Example}

\newtheorem{algorithm}[lemma]{Algorithm}

\newfont{\indic}{bbmss12}

\def\mail#1{\burlalt{#1}{mailto:#1}}

\colorlet{symbols}{blue!90!black}
\colorlet{testcolor}{green!60!black}
\colorlet{connection}{red!30!black}
\def\symbol#1{\textcolor{symbols}{#1}}
\def\symbol#1{\textcolor{symbols}{#1}}

\tikzset{
root/.style={circle,fill=black!50,inner sep=0pt, minimum size=3mm},
        circ/.style={circle,fill=white,draw=black,very thin,inner sep=.5pt, minimum size=1.2mm},
        dot/.style={circle,fill=black,inner sep=0pt, minimum size=1.2mm},
        dotred/.style={circle,fill=black!50,inner sep=0pt, minimum size=2mm},
        var/.style={circle,fill=black!10,draw=black,inner sep=0pt, minimum size=3mm},
        kernel/.style={semithick,shorten >=2pt,shorten <=2pt},
        kernel1/.style={thick},
        kernels/.style={snake=zigzag,shorten >=2pt,shorten <=2pt,segment amplitude=1pt,segment length=4pt,line before snake=2pt,line after snake=5pt,},
		kernels1/.style={snake=zigzag,segment amplitude=0.5pt,segment length=2pt},
		rho1/.style={densely dotted,semithick},
        rho/.style={densely dashed,semithick,shorten >=2pt,shorten <=2pt},
           testfcn/.style={dotted,semithick,shorten >=2pt,shorten <=2pt},
        renorm/.style={shape=circle,fill=white,inner sep=1pt},
        labl/.style={shape=rectangle,fill=white,inner sep=1pt},
        xic/.style={very thin,circle,fill=symbols,draw=black,inner sep=0pt,minimum size=1.2mm},
        xi/.style={very thin,circle,fill=blue!10,draw=black,inner sep=0pt,minimum size=1.2mm},
	xib/.style={very thin,circle,fill=blue!10,draw=black,inner sep=0pt,minimum size=1.6mm},
	xie/.style={very thin,circle,fill=green!50!black,draw=black,inner sep=0pt,minimum size=1mm},
	xid/.style={very thin,circle,fill=symbols,draw=black,inner sep=0pt,minimum size=1.6mm},
	edgetype/.style={very thin,circle,draw=black,inner sep=0pt,minimum size=5mm},
	nodetype/.style={very thick,circle,draw=black,inner sep=0pt,minimum size=5mm},
	kernels2/.style={very thick,draw=connection,segment length=12pt},
clean/.style={thin,circle,fill=black,inner sep=0pt,minimum size=1mm},	not/.style={thin,circle,fill=symbols,draw=connection,fill=connection,inner sep=0pt,minimum size=0.8mm},
	>=stealth,
        }

\tikzset{ individus/.style={scale=0.40,draw,circle,thick,fill=black!10},
 individu/.style={scale=0.40,draw,circle,thick,fill=black!50},       } 
\makeatletter
\def\DeclareSymbol#1#2#3{\expandafter\gdef\csname MH@symb@#1\endcsname{\tikz[baseline=#2,scale=0.15,draw=symbols]{#3}}\expandafter\gdef\csname MH@symb@#1s\endcsname{\scalebox{0.7}{\tikz[baseline=#2,scale=0.15,draw=symbols]{#3}}}}
\def\<#1>{\csname MH@symb@#1\endcsname}
\makeatother

 \def\1{\mathbf{\symbol{1}}}

\usepackage{forest}
\forestset{
	decor/.style = {
		label/.expanded = {[inner sep = 0.2ex, font=\unexpanded{\tiny}]right:{$#1$}}
	},
	root/.style = {minimum size = 0.1ex},
	decorated/.style = {
		for tree = {
			circle, fill, inner sep = 0.3ex, minimum size = 0.4ex,
			edge=thick,
			grow' = north, l = 0, l sep = 1.0ex, s sep = 0.7em,
			fit = tight, parent anchor = center, child anchor = center,
			delay = {decor/.option = content, content =}
		}
	},
	default preamble = {decorated, root},
}

\DeclareMathAlphabet{\mathpzc}{OT1}{pzc}{m}{it}

\def\simnot{\stackrel{\vbox to 0.15em{\hbox{\kern0.07em$^\circ$}}}{\sim}}

\begin{document}

\title{Multi-indice $B$-series }

\author{Yvain Bruned$^1$, Kurusch Ebrahimi-Fard$^2$,  Yingtong Hou$^1$,}
\institute{ 
 IECL (UMR 7502), Université de Lorraine
 \and Norwegian University of Science and Technology NTNU \\
Email:\ \begin{minipage}[t]{\linewidth}
\mail{yvain.bruned@univ-lorraine.fr},
\\
\mail{kurusch.ebrahimi-fard@ntnu.no}
\\
\mail{yingtong.hou@univ-lorraine.fr}.
\end{minipage}}

\maketitle

\begin{abstract}
We propose a novel way to study numerical methods for ordinary differential equations  in one dimension via the notion of multi-indice. The main idea is to replace rooted trees in Butcher's $B$-series by multi-indices. The latter were introduced recently in the context of describing solutions of singular stochastic partial differential equations. The combinatorial shift away from rooted trees allows for a compressed description of numerical schemes. Furthermore, such multi-indices $B$-series uniquely characterize the Taylor expansion of one-dimensional local and affine equivariant maps.
\\[.4em]
\noindent {\scriptsize\textit{MSC classification:} 60L70, 65L06,
16T05}.
\end{abstract}

\setcounter{tocdepth}{2}
\tableofcontents


\section{Introduction}
\label{sec:intro}

Classical $B$-series, also known as Butcher series, play a pivotal role in the analysis of numerical integrators for ordinary differential equations (ODEs). Butcher \cite{Butcher72} initially introduced these series as a vital component of an algebraic framework for integration methods. For a contemporary exposition, one may refer to Butcher \cite{Butcher21}, a modern textbook on the subject, or consult \cite{MMMV17} for a succinct historical overview. Butcher series are built upon the remarkable 1-1 correspondence discovered by Cayley \cite{Cayley1857} between non-planar rooted trees and specific vector fields, termed elementary differentials.

Starting from a general initial value problem 
\begin{equation}
\label{eq:ivp}
	y'=f(y), \quad y(0)=y_0,
\end{equation}
we follow the standard reference \cite{HLW2006} and say that a numerical method
\begin{equation}
\label{eq:numericalmethod}
	y_{k+1}=\Phi(h,f)(y_k)
\end{equation}
of step-size $h$ is a $B$-series method if its Taylor expansion can be expressed as a linear combination of elementary differentials
\begin{equation}
\label{eq:Bseries}
	B_{T}(\alpha, h,f,y) 
	=\sum_{\tau \in T} \frac{h^{|\tau|}\alpha(\tau)}{\sigma(\tau)} F_{f}[\tau](y).
\end{equation}
The sum on the right-hand side runs over the elements from the set $T $ of non-planar rooted trees, i.e., connected and simply connected graphs $\tau \in T$ with vertex set $V(\tau)$ and edge set $E(\tau)$, and a distinguished vertex called the root. Vertices other than the root have exactly one outgoing edge and an arbitrary number of incoming ones; All edges are oriented towards the root vertex. A leaf of a tree is a vertex without any incoming edges. 
The degree $|\tau|$ of a tree $\tau \in T$ is defined in terms of the cardinality of $V(\tau)$. We include the notion of empty tree $e$, which has degree zero, $|e|=0$. The space spanned by all trees $T = \bigcup_{n \ge 0} T_n$, where $T_n$ denotes the set of all trees of degree $n$,  is denoted  $\langle T \rangle$. A rooted tree $\tau$ can be described in terms of a set of rooted trees $\tau_1,\ldots,\tau_n$ and the $B_+$-map which adds a new common root such that $\tau=B_+(\tau_1,\ldots,\tau_n)$ and $|\tau|=1+|\tau_1| + \cdots +|\tau_n|$. A rooted forest is a finite set of rooted trees written as, for example, $\tau_1\ldots\tau_n$. The linear span of the forests is denoted by $ \CF $.    The symmetry factor a rooted tree $\tau \in T$, $\sigma(\tau)$, is defined recursively for $\tau=B_+(\tau_1^{r_1},\ldots,\tau_m^{r_m})$ by
	\begin{equs}
		\sigma(\tau) = \prod_{i}^{m} r_i! \sigma(\tau_i)^{r_i}
	\end{equs}
	where $\tau_1,\ldots,\tau_m$ are distinct and $\tau_i^{r_i}$ means $\tau_i$ appears $r_i$ times. The symmetry factor for a forest is
	\begin{equs}
		\sigma(\tau_1^{r_1}\ldots\tau_m^{r_m}) = \prod_{i}^{m} r_i! \sigma(\tau_i)^{r_i}
	\end{equs} 
	for distinct $\tau_1,\ldots,\tau_m$.
	In the $B$-series, $\alpha: \CF \mapsto \mathbb{R}$ is a linear map from forests to real numbers and it preserves the forest product in the sense that for any $ n \in \mathbb{N}_{+}$
	\begin{equs}
		\alpha(\tau_1\ldots\tau_n) = \prod_{i=1}^n \alpha(\tau_i).
	\end{equs}
	The $B$-series~\eqref{eq:Bseries} describes the exact solution of the ODE~\eqref{eq:ivp} if $\alpha(\tau) = (\tau!) ^{-1}$ (see \cite[Sec.~3.1]{MR2657947}), where the tree factorial $\tau!$ is computed inductively for $\tau=B_+(\tau_1,\dots,\tau_n) \in T$ by defining $\tau! := |\tau| \tau_1 ! \cdots \tau_n !$.  
	The tree factorial of a forest is the product of the factorials of the constituting trees. For the empty tree we define $e !:=1$.

In general, $f : \mathbb{R}^n \to \mathbb{R}^n$ in \eqref{eq:ivp} is a smooth vector field and the elementary differential $F_f$ associates vector fields to trees in an inductive way \cite{HLW2006}, i.e., $F_f[B_+(e)]:=f$ and for a tree $t=B_+(t_1,\ldots,t_n)$ 
\begin{equation} 
\label{eq:elementarydiff} 
	F_f[t]:= f^{(n)}(F_f[t_1],\ldots, F_f[t_n]),
\end{equation}
where $f^{(n)}$ denotes the $n$-th derivative of $f$ -- which is a $n$-linear mapping. As examples, we consider 
\begin{equation*} 
	F_f[\Forest{[[]]}]=f^{(1)}(f)
	\qquad
	F_f[\Forest{[[][]]}]=f^{(2)}(f,f)
	\qquad
	F_f[\Forest{[[[]]]}]=f^{(1)}(f^{(1)}(f)).
\end{equation*}
To get the convergence of the series \eqref{eq:Bseries}, same as in \cite[p.57]{HLW2006}, we assume that $f$ is an analytic function and the linear map $\alpha(\tau)$ satisfies $\alpha(\tau) \lesssim C^{|\tau|}$
for some constant $C$. Then \eqref{eq:Bseries} is bounded by a power series of $Ch$. Therefore, if $h$ sufficiently small, it will converge.

The mapping $F_f$, as defined in \eqref{eq:elementarydiff}, which assigns a specific vector field to each tree, can be more precisely characterised as a pre-Lie morphism. This morphism operates from the free pre-Lie algebra (in a single generator), defined on non-planar rooted trees, to vector fields. To elaborate on this, we first revisit the notion of (left) pre-Lie algebra. Throughout the remainder of the paper, we denote by $\mathbb{K}$ (which is specifically $\mathbb{R}$ in this paper) the base field of characteristic zero, over which all algebraic structures  \cite{cartierpatras2021,manchon2011} are defined. 

\begin{definition}
\label{def:preLie}
A (left) pre-Lie algebra $(\mathfrak g,\triangleright)$ consists of a vector space $\mathfrak g$ with bilinear operation $\triangleright: \mathfrak g \times \mathfrak g \to \mathfrak g$ which satisfies for $x,y,z \in \mathfrak g$ the left pre-Lie identity 
\begin{equation} 
\label{eq:preLie1}
	(x \triangleright y) \triangleright z - x \triangleright (y \triangleright z)
		= (y \triangleright x) \triangleright z - y \triangleright (x \triangleright z)	.
\end{equation}
\end{definition}
Most importantly, the next result shows that pre-Lie algebras are Lie admissible.
\begin{proposition} 
\label{prop:prelie}
Let $(\mathfrak g, \triangleright)$ be a left pre-Lie algebra. The commutator bracket
\begin{equation}
\label{Liebracket}
	\llbracket x,y \rrbracket := x \triangleright y - y \triangleright x 
\end{equation}
defines a Lie algebra on $\mathfrak g$, i.e., it satisfies the Jacobi identity for all $x, y \in \mathfrak g$.  
\end{proposition}

\begin{example}
\label{ex:euclidean}
A prominent example of pre-Lie algebras emerges when examining manifolds that bear a flat and torsion free (linear) connection. 
Remember that a linear connection on a manifold $M$ can be defined on the Lie algebra of its vector fields $\mathfrak X_M$, as a bilinear mapping, that is, $\nabla:\mathfrak X_M \times \mathfrak X_M \rightarrow \mathfrak X_M$, such that $\nabla_{fX}Y=f \nabla_X Y$ and $\nabla_X(fY) = \langle df,X\rangle Y + f \nabla_X Y$, for all $X,Y \in \mathfrak X_M$ and $f \in C^\infty(M)$. Torsion $\mathrm{T}_\nabla$ and curvature $\mathrm{R}_\nabla$ associated to $\nabla$ are tensors defined by 
\begin{equation*}
	\mathrm{T}_\nabla(X,Y)
		:=\nabla_XY-\nabla_YX-[X,Y] \in\mathfrak X_M
\end{equation*}
and 
\begin{equation*}
	\mathrm{R}_\nabla(X,Y)
		:=\nabla_X\nabla_Y-\nabla_Y\nabla_X-\nabla_{[X,Y]} \in\operatorname{End}(TM)
\end{equation*}
for all $X,Y\in\mathfrak X_M$. A linear connection $\nabla$ is \emph{torsion free} if $\mathrm{T}_\nabla(X,Y)=0$ for all $X,Y \in \mathfrak X_M$. If $\mathrm{R}_\nabla(X,Y)=0$, for all $X,Y \in \mathfrak X_M$, then it is called \emph{flat} (or with $0$-curvature). If $\nabla$ is a (linear) torsion-free and flat connection on $M$, then the product $\cdot:\mathfrak X_M \otimes\mathfrak X_M \rightarrow \mathfrak X_M$ defined by $X \cdot Y=\nabla_XY$ is pre-Lie. 
In the sequel, we use the following short hand notation $ \mathfrak X = \mathfrak X_{\mathbb{R}} $.
\end{example}
Another, very related example is a pre-Lie algebra defined on rooted trees. The pre-Lie product is given by grafting on trees, i.e., for two trees $\tau_1,\tau_2 \in T$
\begin{equation*}
	\tau_1 \curvearrowright \tau_2 := \sum_{v \in V(t_2)} \tau_1 \curvearrowright_v \tau_2 \in  \langle T \rangle,
\end{equation*}
where the operation $\tau_1 \curvearrowright_v \tau_2$ is defined by grafting the root of $\tau_1$ via a new edge to the vertex $v \in V(\tau_2)$. For example
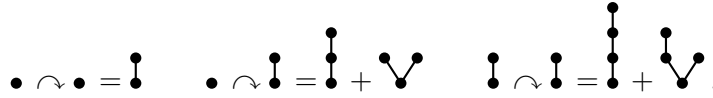
\begin{equation*}
	\Forest{[]} \curvearrowright \Forest{[]} = \Forest{[[]]} 
	\qquad
	\Forest{[]} \curvearrowright \Forest{[[]]} = \Forest{[[[]]]}  + \Forest{[[][]]}  
	\qquad
	\Forest{[[]]} \curvearrowright \Forest{[[]]} = \Forest{[[[[]]]]} + \Forest{[[[]][]]}. 
\end{equation*}
Interestingly, Chapoton and Livernet \cite{ChapLiv2001} showed that $\mathscr{P}(\Forest{[]}\!):=( \langle T \rangle,\curvearrowright)$ is the free pre-Lie algebra in one generator.
Going back to the above Example \ref{ex:euclidean} by considering $M=\mathbb{R}^n$ together with its common flat connection, we see that for vector fields over $\mathbb{R}^n$ 
\begin{equation*}
	f(x):=\sum_{i=1}^n f^i(x) \frac{\partial}{\partial x^i},
\end{equation*}
it is easy to verify that the binary product
\begin{equation*}
		(f \triangleright g) (x) =\sum_{i=1}^n\Big(\sum_{j=1}^n 
	f^j(x)\frac{\partial}{\partial x^j}g^i(x)\Big)\frac{\partial}{\partial x^i}
	=g^{(1)}(f)(x)
\end{equation*}
satisfies the left pre-Lie relation and $F_f$ becomes a pre-Lie morphism
\begin{equation*} 
	F_f[t_1 \curvearrowright t_2 ] = F_f[t_1] \triangleright  F_f[t_2 ].
\end{equation*}

	Conjectured by Dominique Manchon, the product defined in  (\ref{eq:monomialgrafting}) of multi-indices which plays a similar role as the pre-Lie product (grafting) of rooted trees is a Novikov product. It has been made explicit in the second arXiv version of \cite[Lem.~3.5]{Li23} as well as in \cite{BD23}. We therefore recall the notion of Novikov algebra $(N,  \triangleright )$ which is by definition a left pre-Lie algebra and is moreover right-commutative \cite{DL2002}
\begin{equation}
\label{eq:rcom}
	   (x \triangleright y) \triangleright z   = (x \triangleright z)\triangleright y.
\end{equation}

\medskip

Examining a $B$-series \eqref{eq:Bseries} alongside a numerical method \eqref{eq:numericalmethod}, it naturally prompts the question: what specific properties define the latter to possess a $B$-series expansion? This question has persisted as a long-standing problem, only recently resolved in the groundbreaking work by McLachlan et al. \cite{MMMV16} (see also \cite{MMMV17}):

\begin{quote}
``{\it{Numerical methods that can be expanded in $B$-series are defined in all dimensions, so they correspond to sequences of maps—one map for each dimension. A long-standing problem has been to characterise those sequences of maps that arise from $B$-series. This problem is solved here: we prove that a sequence of smooth maps between vector fields on affine spaces has a $B$-series expansion if and only if it is affine equivariant, meaning it respects all affine maps between affine spaces.}}"~\cite{MMMV16}.
\end{quote}

The work by Munthe-Kaas and Verdier \cite{MV16} represents a crucial step in the development of the general result mentioned above. They systematically characterised all local and affine equivariant methods, dubbing them aromatic $B$-series methods, which are essentially generalised $B$-series defined over aromatic trees.  Notably, it was demonstrated \cite[Corollary 8.4]{MV16} that every local and affine equivariant 1-dimensional numerical method inherently possesses a Taylor expansion, taking the form of a proper $B$-series.

\medskip

We now introduce one of the primary objects of this paper: multi-indice $B$-series. They are defined similarly to classical $B$-series, but instead of the index set of rooted trees, $T$, we employ populated multi-indices, $\mathbf{M}$, as the index set that retains only the count of nodes and their arity. Multi-indice $B$-series are of the form
\begin{equs}
	B(a,h,f,y) = a(\emptyset) y + \sum_{z^\beta \in\mathbf{M}} 	
	\frac{h^{|z^\beta|} a(z^\beta)}{S(z^\beta)}F_f[z^\beta](y),
\end{equs}
where $\emptyset$ is the empty forest of populated multi-indices.
See \eqref{multi_indices_B_series} below, where all the notations are properly introduced. 

As with classical $B$-series \cite{CHV05,CHV}, one can define the operations of composition ($\circ$) and substitution ($\circ_s$) for these multi-indice $B$-series. One of our key findings is to show a connection between these analytical operations and two products on multi-indices: $ \star_2 $, the equivalent for multi-indices of the Guin--Oudom (pre-Grossman--Larson) product dual to the Butcher--Connes--Kreimer coproduct \cite{Butcher72,CK,CKI} and $ \star_1 $, the equivalent for multi-indices of the dual of the extraction-contraction coproduct introduced in \cite{CA}.
Specifically, for characters $a$ and $b$ with $a(\emptyset)=b(\emptyset)=1$, and linear map $d(\emptyset) = 0$ with $(c\star_1 d) (\emptyset) = c(\emptyset)$, one has for every $f,g \in \mathcal{C}^{\infty}(\mathbb{R},\mathbb{R})$
	\begin{equs} \label{composition_B}
		B(a,h,f,\cdot)\circ B(b,h,f,y) &= B(b\star_2a,h,f,y),
		\end{equs}
	\begin{equs} \label{substitution_B}
		B(c,h,f,y)	\circ_s B(d,h,g,y) &= B(d\star_1 c, h,g,y).
	\end{equs}
 Here a character is a homomorphism from the forest of multi-indices (defined in  \eqref{forest_product}) to real numbers which preserves the forest product of multi-indices and it send the identity of the forest product (the empty forest $\emptyset$) to the identity of real number multiplication (real number $1$).
This composition \eqref{composition_B} is established in Theorem \ref{thm:main_theorem_composition} and the substitution \eqref{substitution_B} in Theorem \ref{main_theorem_substitution}. The proofs and the algebraic structures involved are inspired by those coming from the rooted trees approach (see \cite{BCCH,BM22,BB21,Li23,B23}).

The concept of ``composition" in \cite[Definition 5.3.]{MV16} (not to be confused with the composition of $B$-series referred to above), is intricately connected to the notion of multi-indice. We are able to refine the main statement in \cite{MV16} by
\begin{theorem} \label{affine_equivariant_multi_indices_intro}
If a map from $\mathcal{C}^{\infty}(\mathfrak X, \mathfrak X)$  with $ \mathfrak X = \mathcal{C}^{\infty}(\mathbb{R}, \mathbb{R})$ is local and affine equivariant, then its Taylor development is a multi-indice $B$-series. Moreover, the choice of the multi-indice $B$-series is unique.
\end{theorem}
The proof of this Theorem can be found in page \pageref{Proof_Theorem_1_5} of this work.

Let us emphasize the significance of this statement. Singular stochastic partial differential equations (SPDEs) are systematically solved through the theory of regularity structures pioneered by Martin Hairer in \cite{reg}. A black box \cite{reg,BHZ,BCCH,ajay} relying on decorated rooted trees, introduced in \cite{BHZ}, covers a large class of these equations (see \cite{FrizHai,BaiHos} for surveys on this topic).
Recently, an alternative combinatorial approach has been proposed \cite{OSSW}, which replaces decorated trees with multi-indices. The concept aims to offer more concise expansions of solutions for these singular dynamics by organizing the expansion according to elementary differentials. Since then, one has seen a fast-growing literature exploring applications as well as the algebraic structures of multi-indices in \cite{BK23,LOTT,T,JZ,BL23,GT}.
For comprehensive surveys on multi-indices, interested readers can refer to \cite{LO23,OST}. 

However, one may wonder whether \textit{multi-indices appeared in the context of numerical analysis prior to their utilization akin to rooted trees in $B$-series, as seen in the work of \cite{OSSW}?} The connection with the work \cite{MV16} gives a positive answer to this question and Theorem \ref{affine_equivariant_multi_indices_intro} shows  that multi-indices are the most natural object for methods in one dimension. This illustrates the importance of multi-indices for both numerical analysis and singular SPDEs. As a final remark, we mention that this could also have implications for low regularity schemes for dispersive equations addressed in \cite{BS}, where decorated trees similar to \cite{BHZ} are used.

\medskip

Let us outline the paper by summarizing the content of its sections. In Section~\ref{sec:multiind}, we begin by revisiting the definition of multi-indices and introduce the pivotal concept of populated multi-indices. These populated multi-indices are endowed with a natural derivation stemming from the free Novikov algebra. From this derivation, one is able to get a product $ \star_2 $ in Definition \ref{def:grafting of multi-indicies} that will be crucially used in the sequel. Notably, this product is equivalent to the Guin--Oudom product for trees. Subsequently, in Algorithm \ref{algorithm:build_tree}  and Corollary \ref{corolla_decomposition}, we unveil a significant property that furnishes a corolla-type decomposition, albeit not uniquely, for populated multi-indices akin to the one observed for trees. Following this crucial observation, we introduce the notion of multi-indice $B$-series, where trees are substituted by populated multi-indices. One of the driving motivations behind exploring these series stems from Proposition \ref{exact_solution_multi_indices}, where precise definitions of coefficients are provided. These coefficients play a key role in demonstrating their correspondence to the Taylor expansion coefficients of the solution of an ODE.

In Section~\ref{sec::composition}, we delve into the composition of multi-indices $B$-series. Initially, we define such composition with a smooth function through (see \eqref{compostion_smooth}). 
In Proposition~\ref{prop:F_derivatives_multiindices}, we demonstrate a crucial morphism property satisfied by the elementary differentials for this product. This proposition is crucial for proving one of our main results, i.e., Theroem~\ref{thm:main_theorem_composition} which provides a precise description of the composition of two multi-indices $B$-series using the product $ \star_2 $. Its proof follows the strategy of proof given in \cite[Thm.~4.6]{B23} which concerns composition of $B$-series in regularity structures. The proof of \cite[Thm.~4.6]{B23} is an adaptation of ideas stemming from \cite{BB21}.
 
In Section~\ref{sec::substitution}, we start by recalling the definition of substitution of $B$-series in \eqref{substitution_def}. Then, we introduce the analoge on multi-indices of the insertion product following the steps of \cite{Li23}. We denote this product by $  \blacktriangleright $ and it is given in Definition \ref{def:insertion_multi-indices}. This definition is very much inspired by the one on trees given in \cite[Sec.~3.4]{BM22} which shows that $  \blacktriangleright $ can be obtained from $ \star_2 $. We then introduce the associative product $ \star_1 $ in Definition \ref{def:star_1_def}. It can also be obtained from the Guin--Oudom procedure. Subsequently, we can present the second main result of this paper, Theorem \ref{main_theorem_substitution} which establishes the connection between the substitution of multi-indices $B$-series and the product $\star_1$. Its proof follows the main argument of \cite[Thm.~4.9]{B23}. One pivotal component in the argument comes from Proposition \ref{prop:morphism}, wherein it is essential to establish a morphism property between $ \star_1 $ and $ \star_2 $, akin to the concept of cointeraction observed between the Butcher--Connes--Kreimer coproduct and the extraction-contraction coproduct on tress described in \cite{CA}.

In Section \ref{sec::affine_equivariant_methods}, we show in Proposition \ref{composition_maps} an explicit link between multi-indices $B$-series and the composition maps defined in \cite[Def.~5.3]{MV16}. The latter were used for obtaining a general statement about aromatic $B$-series and affine equivariant methods which yielded a specific statement in dimension one given in Theorem \ref{local_affine_equivariant}. We recall what an affine equivariant method is in Definition \ref{def_affine_equivariant}. The section closes with our third main result, Theorem \ref{affine_equivariant_multi_indices_intro}, which is a refinement of Theorem \ref{local_affine_equivariant}. The proof of Theorem \ref{affine_equivariant_multi_indices_intro} is based on Proposition \ref{composition_map_populated} that says that a composition map associated to an aromatic tree is a populated multi-indice.


\subsection*{Acknowledgements}
{\small
	Y.~B.~gratefully acknowledges funding support from the European Research Council (ERC) through the ERC Starting Grant Low Regularity Dynamics via Decorated Trees (LoRDeT), grant agreement No.\ 101075208.
	Y.~B.~also would like to thank the Centre for Advanced Study (CAS) at The Norwegian Academy of Science
	and Letters in Oslo, for the nice environment offered during a long stay in September 2023 for the
	programme ”Signature for Images” when this work was started.
	K.~E.~F.~is supported by the Research Council of Norway through project 302831 Computational Dynamics and Stochastics on Manifolds (CODYSMA). He also received support from the "Pure Mathematics in Norway", which is a part of the Mathematics Programme of the Trond Mohn Foundation and would like to thank the Centre for Advanced Study in Oslo for hospitality.}


\section{Multi-indices}
\label{sec:multiind}

The concept of multi-indices emerged initially within the context of studying singular stochastic partial differential equations, as introduced in \cite{OSSW}. It entails a collection of abstract variables $(z_k)_{k \in \mathbb{N}}$. Under a certain assumption (see below), each variable $z_k$ can be seen as corresponding to nodes within a rooted tree possessing precisely $k$ children. Multi-indices, denoted by $\beta : \mathbb{N} \to \mathbb{N}$, encode the occurrence frequency of those variables and thus nodes within a rooted tree. Note that we assume finite support for $\beta$, i.e., $|\{i \in \mathbb{N}\ |\ \beta(i)\neq 0 \}| < \infty$. 

Multi-indices permit a concise representation of rooted trees akin to monomials, capturing the structural aspects of the rooted tree's node distribution.
\begin{equs}
	z^{\beta} : = \prod_{k \in \mathbb{N}} z_k^{\beta(k)}. 
\end{equs}
For instance, we have the following correspondence
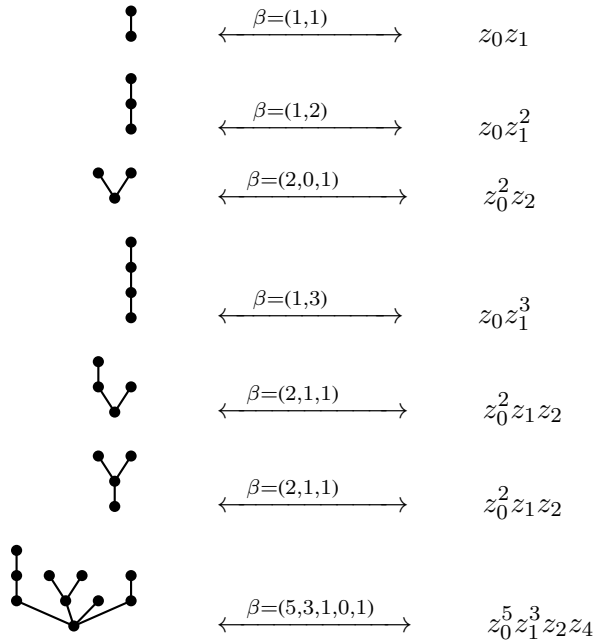
\allowdisplaybreaks
\begin{align*}
	\Forest{[[]]} 	&\qquad\  \xleftrightarrow{\;\;\;\;\beta=(1,1)\;\;\;\;\;\;\;\;\;\;\;}  \qquad\ z_0z_1  				\\[0.2cm]
	\Forest{[[[]]]}   	&\qquad\  \xleftrightarrow{\;\;\;\;\beta=(1,2)\;\;\;\;\;\;\;\;\;\;\;}  \qquad\ z_0z^2_1				\\[0.2cm]
	\Forest{[[][]]}  	&\qquad\  \xleftrightarrow{\;\;\;\beta=(2,0,1)\;\;\;\;\;\;\;\;\;\;}  \qquad\ z^2_0z_2				\\[0.2cm]	
	\Forest{[[[[]]]]} 	&\qquad\  \xleftrightarrow{\;\;\;\;\beta=(1,3)\;\;\;\;\;\;\;\;\;\;\;}  \qquad\ z_0z^3_1				\\[0.2cm]
	\Forest{[[[]][]]}  	&\qquad\  \xleftrightarrow{\;\;\;\beta=(2,1,1)\;\;\;\;\;\;\;\;\;\;}  \qquad\ z^2_0z_1z_2				\\[0.2cm]	
	\Forest{[[[][]]]}  	&\qquad\  \xleftrightarrow{\;\;\;\beta=(2,1,1)\;\;\;\;\;\;\;\;\;\;}  \qquad\ z^2_0z_1z_2				\\[0.2cm]	
	\Forest{[[[[]]][[][]][][[]]]}  	&\qquad\  \xleftrightarrow{\;\;\;\beta=(5,3,1,0,1)\;\;\;\;\;}  \qquad\ z^5_0z^3_1z_2z_4	
\end{align*}
From the above examples, it should be clear that different trees can have the same multi-indice.
In the next definition, we define the so-called ``Counting map" from rooted trees to multi-indices which allows to see the explicit correspondence between them.
\begin{definition}(Counting map)
Consider the general form of a tree $\tau =  B_+(\tau_1,\ldots,\tau_n)$, then we have
\begin{equs}
	\Psi(\bullet) = z_0,
	\quad \Psi(\tau) = z_n \prod_{j=1}^n \Psi(\tau_j).
	\end{equs}
\end{definition}
This map simply counts the frequency of nodes with certain number of children in a tree. However, one can notice that the map $\Psi$ is not surjective as not all multi-indices have related trees. Given our aim to exclusively examine multi-indices corresponding to non-planar rooted trees, we shall focus on those fulfilling the so-called "population" condition \cite{LOT}. This condition stipulates that
\begin{equs}
	\label{populated_1}
	[\beta] :=	\sum_{k \in \mathbb{N}} (1 - k)\beta(k)  = |{\beta}| - \sum_{k \in \mathbb{N}} k \beta(k)  = 1,
\end{equs}
where $|{\beta}|$ is the length of the multi-indice ${\beta} $ defined as
\begin{equs}
	|{\beta}| = \sum_{k \in \mathbb{N}} \beta(k).
\end{equs}
From a tree point of view, $ |{\beta}| $ corresponds to the number of nodes and the sum $ \sum_{j \in \mathbb{N}} j \beta(j) $  corresponds to the number of edges. The population condition captures the feature that the number of edges is $1$ less than the number of nodes of a rooted trees. Let us check the condition on some of the above examples.
\begin{align*}
	[(1,1)] 		&= |{(1,1)}|  - 0 - 1 = 1 \\
	[(1,2)] 		&= |{(1,2)}| - 0 - 2 =  1 \\
	[(5,3,1,0,1)] 	&=|{(5,3,1,0,1)}| - 0 - 3 - 2 - 0 - 4 = 1 .
\end{align*}	 
	Note that we will often abuse terminology and call $z^\beta$ a multi-indice. Therefore, $|z^\beta| := |\beta| $ and $[z^\beta] := [\beta]$. For clarity, we also introduce the notation for the set of populated multi-indices defined as
		\begin{equs}
			\mathbf{M}:= \{z^\beta: [z^\beta] = 1  \}.
		\end{equs}

	One can verify that $z_n \prod_{j=1}^n \Psi(\tau_j)$ is populated for any $\tau_j \in T$ simply by induction, which indicates that image of $\Psi$ could possibly be  the set of populated multi-indices. We would give a rigorous proof of why it is exactly the image through the Algorithm~\ref{algorithm:build_tree} introduced later. In order to show that the algorithm is self-contained, we need the following lemma.
	\begin{lemma}
		\label{construction_tree}
		Let $ z^{\beta} $ be a populated multi-indice, we suppose that $ z^{\beta} = z^{\beta_1} z^{\beta_2} $ where $ z^{\beta_1} $ and $ z^{\beta_2} $ are multi-indices with $  z^{\beta_1} = \prod_{i \in I} z_{k_i} $ and $  z^{\beta_2} = \prod_{j \in J} z_{k_j} $, $I$ and $ J $ being finite sets. We also suppose that we have constructed a decorated tree where the inner nodes are decorated by the variable $ z_{k_i} $, $ i \in I $, with the constraint that at most $ k_i $ edges are incoming to this node. Some of the leaves are decorated by $ z_0 $ but others are not. We denote by $m$ the number of the undecorated leaves of such decorated tree. Then, the cardinal $ | z^{\beta_2} | \geq m $ with equality if  $ z^{\beta_2} = (z_0)^n $ ($n=m$).
	\end{lemma}
	
	\begin{proof}
		By the absurd, suppose that
		\begin{equation*}
			| z^{\beta_2} | < m.
		\end{equation*}
		Since $z^{\beta_1}$ is a tree with $m$ undecorated leaves, the number of nodes is $m-1$ less than the number of edges. Therefore, one has
		\begin{equation*}
			[z^{\beta_1}] = - (m-1).
		\end{equation*}
		On the other hand, we know that $ z^{\beta} $ is populated therefore
		\begin{equation*}
			[z^{\beta}] = 1 = [z^{\beta_1}] + [z^{\beta_2}].
		\end{equation*}
		We use the bound  $ [z_{k_j}] \le 1 $ to notice that
		\begin{equation*}
			[z^{\beta_1}] + [z^{\beta_2}] \le -(m-1) + |z^{\beta_2}| < 1
		\end{equation*}
		which gives a contradiction. If $ z^{\beta_2} = (z_0)^n$, one can easily see that the populated condition can be satisfied only if $ m=n $.
	\end{proof}

	Then we introduce the following algorithm that construct corresponding rooted trees from a given populated multi-indices, which can be viewed as a method to find the pre-image of the counting map $\Psi$.
	\begin{algorithm}\label{algorithm:build_tree} 
		We consider one populated multi-indice $z^\beta \in \mathbf{M}$.
		
		\begin{enumerate}
			\item 	 For $z^\beta = z_0$, one just send it to the tree $\bullet$. Otherwise, one can decompose $ z^{\beta} $ into $ z^{\beta} = z_k z^{\beta'} $, for any $k \in \mathbb{N}_{+}$ such that $\beta(k) \ne 0$. Then, the node $ z_k $ can be interpreted as the root of a decorated tree with $ k $ leaves.
			\vspace{0.5em}
			\item  From Lemma \ref{construction_tree} and the fact that $ z^{\beta} $ is populated, one can find at least $k$ variables in $ z^{\beta'} $ to put on the leaves. We can then choose k nodes decorated by variables $ z_{l_1}, \cdots,  z_{l_k}$ to be attached to each incoming edges to the root.
 			\vspace{0.5em}
			\item  Repeating the operation described in Step 2 by choosing $l_j$ nodes decorated by variables $z_{\cdot}$ from the variables not yet put in the tree and connecting them to the node decorated by $l_j$, one can construct a decorated tree with its root decorated by $ z_k $. 
		\end{enumerate}
	\end{algorithm}
One can notice that in order to have a single rooted tree we cannot  choose $ z_{l_1}, \cdots,  z_{l_k}$ to be all $ z_0 $ as long as the number of variables in $z^{\beta'}$ is larger than the number of nodes to be attached.
\begin{proposition} The algorithm \eqref{algorithm:build_tree} terminates and produces a tree.
	\end{proposition}
\begin{proof}
This is a consequence of  Lemma \ref{construction_tree}, the algorithm terminates when the number of variables left equals the number of nodes needed to be attached to the tree and all variables left are $z_0$, which allows to complete the construction of the tree.
\end{proof}

	Then, one can notice that in the decomposition $ z^{\beta} = z_k \prod_{i=1}^k z^{\beta_i} $, the $ z^{\beta_i} $ are the product of the decoration of the nodes of one subtree incoming the root decorated by $ z_k $.
	We can then directly get the following corollary through the algorithm.
	\begin{corollary}\label{corolla_decomposition}
		For any populated multi-indice $z^\beta  \in \mathbf{M}\setminus \{z_0\}$ and any $n \in \mathbb{N}_{+}$ with $\beta(n) \ne 0$, there exists populated multi-indices $z^{\beta_1},\ldots,z^{\beta_n} \in \mathbf{M}$ such that
		\begin{equs} \label{decomposition_multi}
			z^\beta = z_n\prod_{j=1}^nz^{\beta_j}.
		\end{equs}
	\end{corollary}
	\begin{remark}
		One notices that the decomposition \eqref{decomposition_multi} is not unique as there are choices in the construction of the decorated trees described in Algorithm \ref{algorithm:build_tree}.
	\end{remark}
	
	\begin{example} \label{ex:build_tree}
		We illustrate the construction of a decorated tree from a populated multi-indice by using the Algorithm~\ref{algorithm:build_tree}. Consider the populated multi-indice $z^\beta =  z_0^2 z_1 z_2$ corresponding to $ \beta = (2,1,1) $. One can check that it is populated 
		\begin{equation*}
			[\beta] = (1-0) \beta(0) + (1-1)\beta(1) + (1-2) \beta(2) = 1.
		\end{equation*}	
		We start the algorithm by choosing $ z_2 $, we obtain the following decorated tree
		
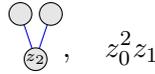
\begin{equs}
			\begin{tikzpicture}[scale=0.2,baseline=-5]
				\coordinate (root) at (0,-1);
				\coordinate (right) at (1,2);
				\coordinate (left) at (-1,2);
				\draw[symbols] (root) -- (right);
				\draw[symbols] (root) -- (left);
				\node[var] (rootnode) at (left) {\tiny{$ $}};
				\node[var] (rootnode) at (right) {\tiny{$  $}};
				\node[var] (rootnode) at (root) {\tiny{$  z_{2} $}};
			\end{tikzpicture}, \quad z_0^2 z_1
		\end{equs}
		with two undecorated leaves.  We use $\bar\beta$ counting the variables that are already used as decorations in the tree we built and $\hat\beta$ storing all variables left, i.e., $\beta = \bar\beta + \hat\beta$. At this point, $ \bar\beta = (0,0,1) $ and $ \hat\beta = (2,1,0) $. Now, we choose in priority two variables different from $ z_0 $ which are $ z_1 $ and $ z_0 $ (there are no other variables). This allows to construct the following decorated trees:
		
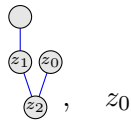
\begin{equs}
			\begin{tikzpicture}[scale=0.2,baseline=-5]
				\coordinate (root) at (0,-1);
				\coordinate (right) at (1,2);
				\coordinate (leftc) at (-1,5);
				\coordinate (left) at (-1,2);
				\draw[symbols] (root) -- (right);
				\draw[symbols] (root) -- (left);
				\draw[symbols] (left) -- (leftc);
				\node[var] (rootnode) at (leftc) {\tiny{$  $}};
				\node[var] (rootnode) at (left) {\tiny{$ z_1 $}};
				\node[var] (rootnode) at (right) {\tiny{$ z_0 $}};
				\node[var] (rootnode) at (root) {\tiny{$  z_{2} $}};
			\end{tikzpicture}, \quad z_0
		\end{equs}
		with one undecorated leaf. At this point, $ \bar\beta = (1,1,1) $ and $ \hat\beta = (1,0,0) $. We terminate the construction of the decorated tree by picking the last variable $ z_0 $
		
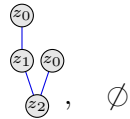
\begin{equs}
			\begin{tikzpicture}[scale=0.2,baseline=-5]
				\coordinate (root) at (0,-1);
				\coordinate (right) at (1,2);
				\coordinate (leftc) at (-1,5);
				\coordinate (left) at (-1,2);
				\draw[symbols] (root) -- (right);
				\draw[symbols] (root) -- (left);
				\draw[symbols] (left) -- (leftc);
				\node[var] (rootnode) at (leftc) {\tiny{$ z_0 $}};
				\node[var] (rootnode) at (left) {\tiny{$ z_1 $}};
				\node[var] (rootnode) at (right) {\tiny{$ z_0 $}};
				\node[var] (rootnode) at (root) {\tiny{$  z_{2} $}};
			\end{tikzpicture}, \quad \emptyset
		\end{equs}
		At this point, $ \bar\beta = (2,1,1) $ and $ \hat\beta = \mathbf{0} $. Now, we can write the following decomposition of $z^\beta$: 
		\begin{equs}
			z^\beta = z_2  z^{\beta_1} z^{\beta_2}, \quad \beta_1 = z_1 z_0, \quad \beta_2 = z_0.
		\end{equs}
		The multi-indices $ z^{\beta_1} $ and $ z^{\beta_2} $ correspond to the following subtrees
		\begin{equs}
			z^{\beta_1} \equiv	\begin{tikzpicture}[scale=0.2,baseline=-5]
				\coordinate (root) at (0,-1);
				\coordinate (rootc) at (0,2);
				\draw[symbols] (root) -- (rootc);
				\node[var] (rootnode) at (rootc) {\tiny{$ z_0 $}};
				\node[var] (rootnode) at (root) {\tiny{$ z_1 $}};
			\end{tikzpicture}, 
			\quad z^{\beta_2} \equiv
			\begin{tikzpicture}[scale=0.2,baseline=-5]
				\coordinate (root) at (0,-1);
				\node[var] (rootnode) at (root) {\tiny{$ z_0 $}};
			\end{tikzpicture}.
		\end{equs}
		The non-uniqueness of the construction becomes clear from another possible choice to start the algorithm, i.e., by choosing $z_1$ as the decoration of the root first which would result in another tree compatible with the multi-indice $ \beta = (2,1,1) $.  
		Also, one observes that it is crucial to have the constraint that we cannot choose $ z_{l_1}, \cdots,  z_{l_k}$ to be all $ z_0 $ as long as the number of variables in $z^{\beta'}$ is larger than the number of nodes to be attached" in the second step of Algorithm \ref{algorithm:build_tree}. Indeed, if after the first step, we have chosen twice $ z_0 $, we will have obtained
		\begin{equs}
			\begin{tikzpicture}[scale=0.2,baseline=-5]
				\coordinate (root) at (0,-1);
				\coordinate (right) at (1,2);
				\coordinate (left) at (-1,2);
				\draw[symbols] (root) -- (right);
				\draw[symbols] (root) -- (left);
				\node[var] (rootnode) at (left) {\tiny{$ z_0 $}};
				\node[var] (rootnode) at (right) {\tiny{$ z_0 $}};
				\node[var] (rootnode) at (root) {\tiny{$  z_{2} $}};
			\end{tikzpicture}, \quad  z_1.
		\end{equs}
		The construction of the decorated tree is completed but we have one variable $ z_1 $ left. 
	\end{example}
The Algorithm also leads to the following theorem.
\begin{theorem}
	For every populated multi-indice $z^\beta$, there exists at least one tree $\tau$ such that for each $k \in \mathbb{N}$
	\begin{equs}
		\Psi(\tau) = z^{\beta}.
	\end{equs}
\end{theorem}
\begin{proof}
	This is an immediate consequence of Algorithm \ref{algorithm:build_tree} as for any populated multi-indices we can obtain its associated rooted tree(s) inductively through the algorithm.
\end{proof}
This theorem together with the fact that $z_n \prod_{j=1}^n \Psi(\tau_j)$ is populated for any $T \ni\tau = B_{+}(\tau_1 \cdots \tau_n)$  indicates that populated multi-indices are the image of the counting map $\Psi$.

\vspace{0.5em}	
After introducing multi-indices and their connections to rooted trees we can explore some algebraic structure of them. We firstly define a product $ \triangleright $ on the $\mathbb{K}$-linear span of populated multi-indices
\begin{equs}
\label{eq:monomialgrafting}
	z^{\beta}\triangleright  z^{\beta'} := z^{\beta}D(z^{\beta'}),
\end{equs}
where $D$ is the derivation given by
\begin{equs}
	D = \sum_{k \in \mathbb{N}} z_{k+1} \partial_{z_k}.  
\end{equs}
Here, $ \partial_{z_k} $ is the ordinary partial derivative in the coordinate $ z_k $.

Assuming that $z^{\beta'}$ is a populated multi-indice, we see immediately for the multi-indice $z^\beta$  with {$\beta = (1)$} that the product  \eqref{eq:monomialgrafting}
$$
	z^{(1)} \triangleright z^{\beta'} = {z_0}D(z^{\beta'})
$$ 
amounts to adding a child to one of the nodes of a rooted tree corresponding to multi-indice $z^{\beta'}$ in all possible ways. As an example, we observe that
$$
	{z_0}D(z_0z_1) = {z_0}z_1^2 + {z^2_0}z_2
	\qquad
	{z_0}D(z_0z^2_1) = {z_0}z^3_1 + 2 {z_0^2}z_1z_2.
$$
For populated multi-indice, the derivation $D$ can be interpreted as adding an edge to a vertex, which changes the latter from having, say, $k$ children to having $k+1$ children. Hence, more generally, for populated multi-indices the product \eqref{eq:monomialgrafting} plays a role analog to (pre-Lie) grafting of trees onto trees. 
\begin{proposition}
	The map $ \Psi $ is actually a pre-Lie morphism in the sense that one has for every trees  $ \tau_1 $ and $ \tau_2 $
	\begin{equs}
		\Psi \left( 	\tau_1 \curvearrowright \tau_2 \right) 
		= \Psi \left( 	\tau_1 \right)  \triangleright  \Psi \left(  \tau_2 \right)
		=\Psi \left( 	\tau_1 \right)  D  \Psi \left(  \tau_2 \right).
	\end{equs}
\end{proposition}

		\begin{proof}
			For any rooted trees $\tau_1$ and $\tau_2$, by the definition of the grafting product and the linearity of $\Psi $ 
			\begin{equs}
				\Psi \left( 	\tau_1 \curvearrowright \tau_2 \right) = \sum_{v \in V(\tau_2)} \Psi \left( \tau_1 \curvearrowright_v \tau_2\right).
			\end{equs}
			Suppose that we have chosen one specific node $v \in V(\tau_2)$ having $k$ children before grafting, $\curvearrowright_v$ will send the arity of $v$ from $k_v$ to $k_v+1$. The arity of other nodes in $\tau_1$ and $\tau_2$ will not be changed. Since $\tau_1$ and $\tau_1$ are connected through this newly added edge, 
			\begin{equs}
				\Psi \left( \tau_1 \curvearrowright_v \tau_2\right) = \Psi \left( 	\tau_1 \right) \Psi \left(  \tau_2 \right)\frac{z_{k_v+1}}{z_{k_v}}.
			\end{equs}
			Finally, we have
			\begin{equs}
				\Psi \left( 	\tau_1 \curvearrowright \tau_2 \right) &= \sum_{v \in V(\tau_2)} \Psi \left( \tau_1 \curvearrowright_v \tau_2\right)
				\\&=
				\Psi \left( 	\tau_1 \right)\sum_{k \in \mathbb{N}} z_{k+1} \partial_{z_k}\Psi \left(  \tau_2 \right)
				\\&=\Psi \left( 	\tau_1 \right)  D  \Psi \left(  \tau_2 \right).
			\end{equs}
		\end{proof}

We introduce the commutative forest product $\tilde{\prod}_{j=1}^n z^{\beta_j}$ (or alternatively denoted as $z^\alpha\, \tilde{\bullet}\, z^\beta$ ) for multi-indices. Similar to the forest product of trees, the product $\tilde{\prod}_{j=1}^n z^{\beta_j}$ amounts to  the juxtaposition of multi-indices $z^{\beta_j}$ without merging the frequencies $\beta_j$, i.e.,
\begin{equs} \label{forest_product}
	\tilde{\prod}_{j=1}^n z^{\beta_j} \ne z^{\sum_{j=1}^n \beta_j}.
\end{equs}  
Notice that there is no order among the $z^{\beta_j}$ since the forest product is commutative. 
We will use $\CM$ to denote the set of all forests of populated multi-indices. 
\begin{equs}
	\CM := \left\{\tilde{\prod}_{j=1}^n z^{\beta_j}: z^{\beta_j} \in \mathbf{M}, \,n \in \mathbb{N}  \right\}.
\end{equs}
In particular the empty forest $\emptyset \in \CM$ obtained when $n=0$ is the identity element of the forest product.
	For clarity, we will also use the following notations for some subsets of $\CM $.
	For any $n \in  \mathbb{N}$,
	\begin{equs}
		\CM_n := \left\{\tilde{\prod}_{j=1}^n z^{\beta_j}: z^{\beta_j} \in \mathbf{M}\right\}.
	\end{equs}
	We also introduce notation for the set of all non-empty forests of populated multi-indices
\begin{equs}
	\CM_{+} := \left\{\tilde{\prod}_{j=1}^n z^{\beta_j}: z^{\beta_j} \in \mathbf{M}, \,n \in \mathbb{N}_{+}  \right\}.
\end{equs}
To further simplify the notation we use $\tilde{z}^{\tilde{\beta}}$ denoting a forest of multi-indices, where $\tilde{\beta}$ is a collection of populated multi-indices without order.  For $\tilde{\beta} = \{\beta_1,\ldots,\beta_n\}$
\begin{equs}
	\tilde{z}^{\tilde{\beta}} := 	\tilde{\prod}_{j=1}^n z^{\beta_j}.
\end{equs}
The repetition of individual populated multi-indices in $\tilde{\beta}$ is allowed, which means $z^{\beta_1},\ldots,z^{\beta_n}$ are not necessarily to be distinct.
Then the length of a forest is simply the sum of the length of individual multi-indices consisting of it. i.e.,
\begin{equs}
	|\emptyset| = 0 \quad \text{and} \quad |\tilde{z}^{\tilde{\beta}}| = \sum_{j=1}^n|z^{\beta_j}| \quad \text{ for any forests }  \tilde{z}^{\tilde{\beta}} = 	\tilde{\prod}_{j=1}^n z^{\beta_j}.
	\end{equs}

One can notice that the derivation $D$ operates on a non-empty forest as the following:
for any $ \CM_{+}  \ni \tilde{z}^{\tilde{\beta}} = 	\tilde{\prod}_{j=1}^n z^{\beta_j}$,
	\begin{equs}
		D	\tilde{z}^{\tilde{\beta}} = \sum_{j=1}^n   \left(\tilde{\prod}_{i\ne j} z^{\beta_i}\right)\, \tilde{\bullet} \, Dz^{\beta_j},
	\end{equs}
	which means it obeys the Leibniz rule. 
	In the next definition, we introduce $\star_2$ a Guin--Oudom type product of forests of populated multi-indices, which is the counterparty to the Guin--Oudom grafting product of forests of trees. This means $\star_2$ can also be derived from the Guin--Oudom procedure (see \cite{GD,Guin1}) applied to the pre-Lie product $\triangleright$ defined in \eqref{eq:monomialgrafting} -- turning the span of multi-indices into a Novikov algebra.  It would be the core of the composition law of Multi-indice $B$-series.
\begin{definition}
\label{def:grafting of multi-indicies}
For the empty forest $\emptyset$ and a non-empty forest of populated multi-indices $\tilde{z}^{\tilde{\alpha}} \in \CM_{+}$ we define a product $\star_2$
\begin{equs}
	\emptyset\star_2\emptyset := \emptyset, 
	\quad \emptyset\star_2\tilde{z}^{\tilde{\alpha}} := \tilde{z}^{\tilde{\alpha}}, 
	\quad \text{and} \quad
	\tilde{z}^{\tilde{\alpha}}\star_2\emptyset := \tilde{z}^{\tilde{\alpha}}.
\end{equs}	
For $\CM_{+} \ni \tilde{z}^{\tilde{\beta}} = \tilde{\prod}_{j=1}^nz^{\beta_j}$ and $\tilde{z}^{\tilde{\alpha}} \in \CM_{+}$	we define $\star_2$ as
\begin{equs}
\label{eq:GLprod}
 	\tilde{\prod}_{j=1}^nz^{\beta_j}\star_2\tilde{z}^{\tilde{\alpha}} 
	:= \left(\prod_{j=1}^nz^{\beta_j}\right)D^n\tilde{z}^{\tilde{\alpha}}.
\end{equs}

\end{definition} 
\begin{remark} 
\label{rmk:GuinOudom}
Using the Guin--Oudom functor, one gets for any decorated trees $\tau_j$ and $\mu$,
\begin{equs} \label{eq:GuinOudom}
	\Psi\left(\tilde{\prod}_{j=1}^n \tau_j \star_{\tiny{\text{GO}}} \mu \right) = \tilde{\prod}_{j=1}^n \Psi(\tau_j) \star_2 \Psi(\mu),
\end{equs}
where $ \tilde{\prod}_{j=1}^n \tau_j  $ corresponds to the forest product for the trees and $ \star_{\tiny{\text{GO}}} $ is the Guin--Oudom product of trees which is obtained from the earlier defined grafting product $ \curvearrowright $ extended to forests (Guin--Oudom construction \cite{GD,Guin1} on the pre-Lie product  $ \curvearrowright $).
By applying the unshuffle coproduct to the Guin--Oudom type product $\star_2$, one can also obtain a Grossman--Larson type product similar to the more familiar Grossman--Larson product $\star_{\tiny{\text{GL}}}$  defined on forests of non-planar rooted trees. 
If we change $\star_{\tiny{\text{GO}}}$ to $\star_{\tiny{\text{GL}}}$ on the left-hand side of \eqref{eq:GuinOudom} and change $\star_2$ to the Grossman--Larson product of multi-indices, the morphism property is recovered. However, we will not need the Grossman--Larson product in the sequel.
\end{remark}

\bigskip

We now introduce one of the main objects of this paper: multi-indice $B$-series. They are defined as
\begin{equs} \label{multi_indices_B_series}
	B(a,h,f,y) = a(\emptyset) y + \sum_{z^\beta \in\mathbf{M} } 	
	\frac{h^{|z^\beta|} a(z^\beta)}{S(z^\beta)}F_f[z^\beta](y),
\end{equs}
where 
\begin{itemize}
	\item the map $ a $ is a linear map from $ \CM$ into $ \mathbb{R}$,
 which preserves the multiplicativity of the forest product
		\begin{equs}\label{character}
		a(\emptyset \tilde{\bullet} \tilde{z}^{\tilde{\beta}}) =   a(\emptyset)a(\tilde{z}^{\tilde{\beta}}) , \quad	a\left(\tilde{\prod}_{j=1}^nz^{\beta_j}\right) = \prod_{j=1}^n a(z^{\beta_j})
		\end{equs}
		for any forest $\tilde{z}^{\tilde{\beta}} \in \CM$ and any $z^{\beta_j} \in \mathbf{M}$ . 
		Therefore, if $a(\emptyset)=1$, $ a $ is a character of multi-indices with respect to the forest product
		 as $\emptyset$ is the identity in the forest product and $1$ is the identity in the real number multiplication.
		\item $ S(z^\beta) $ is the symmetry factor of the multi-indice  $ z^\beta$ given by
		\begin{equs}
			S(z^\beta) := \prod_{k \in \mathbb{N}}\left(k!\right)^{\beta(k)}
		\end{equs}
		and can be extended to the forest of multi-indices by
		\begin{equs}
		S\left(\tilde{\prod}_{j=1}^n\left(z^{\beta_j}\right)^{r_j}\right) := \prod_{j=1}^nr_j!S\left(z^{\beta_j}\right)^{r_j},
\end{equs}
		where $z^{\beta_j}$ are disjoint. Accordingly, the inner product is defined as
	\begin{equs}
			<z^\alpha, z^\beta> = \delta_{\tiny \alpha, \beta}S(z^\alpha), \quad
	\text{where } \quad 
	\delta_{\tiny \alpha, \beta} = 
	\begin{matrix}	
		\begin{cases}
			1, \text{if } \alpha = \beta \\
			
			0, \text{if } \alpha \ne \beta \\
		\end{cases}
		
	\end{matrix}.
\end{equs}

		\item $ F_f[z^\beta] $ are elementary differentials defined by 
		\begin{equs}
			&F_f[z^\beta](y) := \prod_{k \in \mathbb{N}}\left( f^{(k)}(y)\right)^{\beta(k)}, 
	\end{equs}
	for $z^\beta \in \mathbf{M}$ and with $f : \mathbb{R} \mapsto \mathbb{R} $ and $ f^{(k)} $ is the $k$-th derivative of $ f $.  Moreover, $F_f$ is linear in populated multi-indices and for a non-empty forest $\tilde{z}^{\tilde{\beta}} = \tilde{\prod}_{j=1}^nz^{\beta_j} $ we have $ \prod_{j=1}^n F_f[z^{\beta_j}]$.
	\item  similar to the classical rooted tree $B$-series, for the convergence of \eqref{multi_indices_B_series}, we assume that $f$ is an analytic function and the linear map $\alpha(z^\beta)$ satisfies 
$
		\alpha(z^\beta) \lesssim C^{|z^\beta|}
$
	for any $z^\beta \in \mathbf{M}$ and some constant $C$. Then \eqref{multi_indices_B_series} is bounded by a power series of $Ch$. Therefore, if $h$ sufficiently small it will converge.
\end{itemize}

\smallskip

Like for Butcher's $B$-series built on rooted trees, there exists a choice of the linear map $ a $ such that the multi-indices $B$-series is the exact solution of the 1-dimensional initial value problem:
\begin{equs} 
\label{main_equation}
	y' = f(y), \quad y(0) = y_0 \in \mathbb{R}.
	\end{equs}
This is the object of the next proposition:

\begin{proposition} 
\label{exact_solution_multi_indices}
We suppose that $ f \in \mathcal{C}^{\infty}(\mathbb{R},\mathbb{R}) $ is an analytic function. Then the exact solution of \eqref{main_equation} is given by a multi-indice $B$-series with a linear map $ a $ given by:
\begin{equs} 
\label{condition_1}
	a(  z^{\beta} ) 
	= \frac{1}{|z^{\beta}|} \sum_{z^{\beta} 
	= z_k \prod_{i=1}^k z^{\beta_i}} \prod_{i=1}^k a(z^{\beta_i})
\end{equs}
 where $\sum_{z^{\beta} 
	= z_k \prod_{i=1}^k z^{\beta_i}}$ runs over all $k \in \mathbb{N}$ and $z^{\beta_i} \in \mathbf{M}$. 
\end{proposition}
It is obvious that $a$ satisfies the assumption on the geometric growth condition of the linear map, which together with the analytic $f$ ensures the convergence of this multi-indice series.
\begin{proof}
We consider $ y $ given by
\begin{equs}
	y(h) = B(a,h,f,y_0).
\end{equs}
Computing its derivative, on gets
\begin{equs}
	y'(h) = \sum_{z^\beta \in\mathbf{M} } |z^\beta| \frac{h^{|z^\beta|-1} 
		a(z^\beta)}{S(z^\beta)}F_f[z^\beta](y_0).
\end{equs}
On the other side, if we plug $ y(h) $ inside $ f $ and we assumed that $f$ is analytic which allows us to perform a Taylor expansion:
\begin{equs}
	f(y(h)) 
	& = \sum_{k \in \mathbb{N}} \frac{1}{k!} f^{(k)}(y_0) (y(h)-y_0)^k \\
	&  = \sum_{k \in \mathbb{N}} \frac{1}{k!} f^{(k)}(y_0)	
		\left(\sum_{z^\beta \in \mathbf{M}}\frac{a(z^\beta)h^{|z^\beta|}}{S(z^\beta)}F_f[z^\beta](y_0)\right)^k\\ 
	& = 	\sum_{k \in \mathbb{N}}
		\sum_{z^{\beta_1},\ldots,z^{\beta_k} \in \mathbf{M}} \frac{1}{k!} f^{(k)}(y_0)\prod_{i=1}^k
	\left(\frac{a(z^{\beta_i})h^{|z^{\beta_i}|}}{S(z^{\beta_i})}F_f[z^{\beta_i}](y_0)\right).
	\end{equs}
One notices that by multiplicativity of the maps $ F_f[\cdot]$ and $S(\cdot)$, one has
\begin{equs}
 	k!\prod_{i=1}^k S(z^{\beta_i}) 
		&= S(  z_{k} \prod_{i=1}^k z^{\beta_i} ), \quad S(z_k) = k!, \\
 	f^{(k)} \prod_{i=1}^kF_f[z^{\beta_i}] 
	& = F_f[z_k \prod_{i=1}^k z^{\beta_i}], \quad f^{(k)} = F_f[z_k].
	\end{equs}
Observe that $ z_k \prod_{i=1}^k z^{\beta_i} $ is a populated multi-indice.
Moreover, one has
\begin{equs}
	\sum_{j=1}^k  |z^{\beta_i}| = |  z_k \prod_{i=1}^k z^{\beta_i} | -1.
\end{equs}
In the end, we get the following condition on the linear map $ a $:
\begin{equs}
	a(  z^{\beta} ) = \frac{1}{|z^{\beta}|} \sum_{z^{\beta} = z_k \prod_{i=1}^k z^{\beta_i}} \prod_{i=1}^k a(z^{\beta_i})
\end{equs}
where the $ \beta_i $ are populated multi-indices.
\end{proof}

\begin{remark}
The condition observed is in agreement with the one found in the context of singular SPDEs in \cite{LOT} and in Rough Paths in \cite{Li23}. 
\end{remark}

\section{Composition of multi-indice $B$-series}
\label{sec::composition}

We define the composition of a $B$-series $B(a,h,f,y)$ with a smooth function $ g \in \mathcal{C}^{\infty}(\mathbb{R},\mathbb{R})$ as 
\begin{equs} \label{compostion_smooth}
	g(B(a,h,f,y)) 
\end{equs}
Consequently, the composition of two multi-indices $B$-series is given by
\begin{equs}
	B(a,h,f,\cdot)\circ B(b,h,g,y) &= B(a,h,f,B(b,h,g,y)),
\end{equs}
Then as we have assumed that $f,g$ are analytic and $a,b$ has geometric growth bounded by the length of multi-indices, we can get the Taylor expansion form
\begin{equs}
	B(a,h,f,\cdot)\circ B(b,h,g,y) = \sum_{k \in \mathbb{N}} \frac{1}{k!}\partial^kB(a,h,f,y)(B(b,h,g,y)-y)^k.
\end{equs}
which is a convergent series since in Theorem  \ref{thm:main_theorem_composition} we can see if $a,b$ satisfy the growth condition the coefficient (in the form of $a(\cdot)b(\cdot)$) in the composed series also satisfies.

The following proposition illustrating the morphism property of elementary differentials with respect to the product $\star_2$ is necessary for the proof of the composition law.
\begin{proposition} 
\label{prop:F_derivatives_multiindices}
For every $n\in \mathbb{N}_{+}$, $z^{\beta_j} \in \mathbf{M}$ and $z^\alpha \in \mathbf{M}$, one has
\begin{equs}\label{grafting}
	F_f\left[\tilde{\prod}_{j=1}^nz^{\beta_j}\star_2z^\alpha\right] = \left(\prod_{j=1}^nF_f[z^{\beta_j}]\right)\partial^nF_f[z^\alpha].
\end{equs}		
\end{proposition}

\begin{proof}
By the Leibniz rule
	\begin{equs}
		\partial^nF_f[z^\alpha]
		= \sum_{\sum_{k \in \mathbb{N}}l_k = n}{n \choose l_{k_1},\ldots,l_{k_m}}
		\prod_{k \in \mathbb{N}}\partial^{l_k}
		\left(f^{(k)}\right)^{\alpha(k)}, 
	\end{equs}
where $l_{k_1},\ldots,l_{k_m}$ is the collection of all non-zero terms among $l_k$. Since $n$ is finite and $\alpha$ has finite support, the above partition of $n$ is finite and thus $m$ is finite. Then, regard $\left(f^{(k)}\right)^{\alpha(k)}$ as the $\alpha(k)$-multiple product of $f^{(k)}$ and apply the Leibniz rule again.
	\begin{equs}
		\partial^nF_f[z^\alpha] 
		=&
	\sum_{\sum_{k \in \mathbb{N}}l_k = n}{n \choose l_{k_1},\ldots,l_{k_m}}
		\\& 
		\prod_{k \in \mathbb{N}}\left(
		\sum_{v_1^k+ \cdots +v_{\alpha(k)}^k = l_k}{l_k \choose v_1^k,\ldots,v_{\alpha(k)}^k}
		\prod_{i=1}^{\alpha(k)}\partial^{v_i^k}f^{(k)}\right).
	\end{equs}
For the left-hand side of Equation \eqref{grafting}, 
	\begin{equs}
		F_f\left[\tilde{\prod}_{j=1}^nz^{\beta_j}\star_2z^\alpha\right] 
		= F_f\left[\left(\prod_{j=1}^nz^{\beta_j}\right)D^nz^\alpha\right].
	\end{equs}
Regarding $z_k^{\alpha(k)}$ as a $\alpha(k)$-multiple product of $z_k$ and applying the Leibniz rule leads to
	\begin{equs}
		D^nz^\alpha 
		=& 
		\sum_{\sum_{k \in \mathbb{N}}l_k = n}{n \choose l_{k_1},\ldots,l_{k_m}}	
		\\&
		\prod_{k \in \mathbb{N}}\left(
		\sum_{v_1^k+ \cdots +v_{\alpha(k)}^k = l_k}{l_k \choose v_1^k,\ldots,v_{\alpha(k)}^k}
		\prod_{i=1}^{\alpha(k)}D^{v_i^k}z_k\right)
		\\=&
		\sum_{\sum_{k \in \mathbb{N}}l_k = n}{n \choose l_{k_1},\ldots,l_{k_m}}	
		\\&
		\prod_{k \in \mathbb{N}}\left(
		\sum_{v_1^k+ \cdots +v_{\alpha(k)}^k = l_k}{l_k \choose v_1^k,\ldots,v_{\alpha(k)}^k}
		\prod_{i=1}^{\alpha(k)}z_{k+v_i^k}\right).	
	\end{equs}
Therefore,
	\begin{equs}
		F_f\left[\tilde{\prod}_{j=1}^nz^{\beta_j}\star_2z^\alpha\right]
		=&
		\left(\prod_{j=1}^nF_f[z^{\beta_j}]\right)\sum_{\sum_{k \in \mathbb{N}}l_k = n}{n \choose l_{k_1},\ldots,l_{k_m}}
		\\&
		\prod_{k \in \mathbb{N}}\left(
		\sum_{v_1^k+ \cdots +v_{\alpha(k)}^k = l_k}{l_k \choose v_1^k,\ldots,v_{\alpha(k)}^k}
		\prod_{i=1}^{\alpha(k)}F_f\left[z_{k+v_i^k}\right]\right)
		\\=& 
		\left(\prod_{j=1}^nF_f[z^{\beta_j}]\right)\sum_{\sum_{k \in \mathbb{N}}l_k = n}{n \choose l_{k_1},\ldots,l_{k_m}}
		\\&        
		\prod_{k \in \mathbb{N}}\left(
		\sum_{v_1^k+ \cdots +v_{\alpha(k)}^k = l_k}{l_k \choose v_1^k,\ldots,v_{\alpha(k)}^k}
		\prod_{i=1}^{\alpha(k)}\partial^{v_i^k}f^{(k)}\right)
	\end{equs}
	which concludes the proof.
\end{proof}
One can see that this proposition can be extended from $z^\alpha \in \mathbf{M}$ to $\tilde{z}^{\tilde{\alpha}} \in \CM_{+}$, as for a forest $\tilde{z}^{\tilde{\alpha}} = \tilde{\prod}_{j=1}^n z^{\alpha_j}$ we have $F_f[\tilde{z}^{\tilde{\alpha}}] = \prod_{j=1}^n F_f[z^{\alpha_j}]$.  Since both $D^n$ applied to a forest and $\partial^n$ applied to the monomial of elementary differentials obey the Leibniz rule the proof works for forest cases. 
	\begin{remark}
		This proposition can also be proved by adapting the universality property to \cite[Eq.~(6.21), (6.22)]{LOT}.
	\end{remark}

Suppose that $\Delta_2$ is the dual coproduct of $\star_2$, which means they have the adjoint relation
\begin{equs} \label{adjoint_2}
	<\tilde{z}^{\tilde{\beta}} \otimes \tilde{z}^{\tilde{\alpha}} , \Delta_2 \tilde{z}^{\tilde{\mu}} > 
	= <\tilde{z}^{\tilde{\beta}}  \star_2 \tilde{z}^{\tilde{\alpha}} , \tilde{z}^{\tilde{\mu}} >
\end{equs}
for $\tilde{z}^{\tilde{\beta}}, \tilde{z}^{\tilde{\alpha}} , \tilde{z}^{\tilde{\mu}} \in \CM$.
In order to show the well-posedness of $\Delta_2$, we need the following lemmas.
\begin{lemma} \label{lemma:length_sum}
	If $<\tilde{z}^{\tilde{\beta}}  \otimes  \tilde{z}^{\tilde{\alpha}}, \Delta_2 \tilde{z}^{\tilde{\mu}}> \ne 0$, then $|\tilde{z}^{\tilde{\beta}} | + | \tilde{z}^{\tilde{\alpha}}| = | \tilde{z}^{\tilde{\mu}}|$, for any $\tilde{z}^{\tilde{\beta}}, \tilde{z}^{\tilde{\alpha}} , \tilde{z}^{\tilde{\mu}} \in \CM$.
\end{lemma}
\begin{proof}
	We first write the coproduct $\Delta_2$ and the product $\star_2$ in the following forms.
	\begin{equs}
		& \Delta_2 \tilde{z}^{\tilde{\mu}} = \sum_{\tilde{z}^{\tilde{\beta}'}, \tilde{z}^{\tilde{\alpha}'} \in \CM} C_{\Delta_2}(\tilde{z}^{\tilde{\beta'}}, \tilde{z}^{\tilde{\alpha}'},  \tilde{z}^{\tilde{\mu}}) \tilde{z}^{\tilde{\beta}'}  \otimes  \tilde{z}^{\tilde{\alpha}'}
    	\\& \tilde{z}^{\tilde{\beta}}  \star_2 \tilde{z}^{\tilde{\alpha}} = \sum_{\tilde{z}^{\tilde{\mu}'}\in \CM} C_{\star_2}(\tilde{z}^{\tilde{\beta}}, \tilde{z}^{\tilde{\alpha}},  \tilde{z}^{\tilde{\mu}'})\tilde{z}^{\tilde{\mu}'},
	\end{equs}
	where $C_{\Delta_2} \in \mathbb{N}$ and $C_{\star_2} \in \mathbb{N}$. 
	Then the inner products can be rewritten as
	\begin{equs} \label{inner_2}
		& <\tilde{z}^{\tilde{\beta}} \otimes \tilde{z}^{\tilde{\alpha}} , \Delta_2 \tilde{z}^{\tilde{\mu}} >  = C_{\Delta_2}(\tilde{z}^{\tilde{\beta}}, \tilde{z}^{\tilde{\alpha}},  \tilde{z}^{\tilde{\mu}}) S(\tilde{z}^{\tilde{\beta}})S(\tilde{z}^{\tilde{\alpha}})
		\\& <\tilde{z}^{\tilde{\beta}}  \star_2 \tilde{z}^{\tilde{\alpha}} , \tilde{z}^{\tilde{\mu}} > =  C_{\star_2}(\tilde{z}^{\tilde{\beta}}, \tilde{z}^{\tilde{\alpha}},  \tilde{z}^{\tilde{\mu}}) S(\tilde{z}^{\tilde{\mu}}).
	\end{equs}
	By the adjoint relation \eqref{adjoint_2}, $<\tilde{z}^{\tilde{\beta}}  \otimes  \tilde{z}^{\tilde{\alpha}}, \Delta_2 \tilde{z}^{\tilde{\mu}}> \ne 0$ implies $C_{\star_2}(\tilde{z}^{\tilde{\beta}}, \tilde{z}^{\tilde{\alpha}},  \tilde{z}^{\tilde{\mu}}) \ne 0$, which means  $\tilde{z}^{\tilde{\mu}}$ can be obtained by ``grafting" $\tilde{z}^{\tilde{\beta}} $ to $\tilde{z}^{\tilde{\alpha}}$ through $\star_2$. In the definition of $\star_2$ we can see that $|\tilde{z}^{\tilde{\beta}} |$ is preserved through the monomial multiplication. The only thing left is the deviation $D$ that amounts to replace one $z_k$ by one $z_{k+1}$, which preserves $|\tilde{z}^{\tilde{\alpha}}|$.
\end{proof}
\begin{lemma}\label{lemma:length_finite}
	For any non-empty forest $\tilde{z}^{\tilde{\beta}} \in \CM_{+}$, if $|\tilde{z}^{\tilde{\beta}}|$ is fixed and finite, the number of possible forms  $\tilde{z}^{\tilde{\beta}} \in \CM$ can take is finite.
\end{lemma}
\begin{proof}
 Firstly, for an individual populated multi-indice $z^\beta \in \CM_1 = \mathbf{M}$, from the population condition \eqref{populated_1} and $\beta(k) \ge 0$ we have 
 \begin{equs}
 	|z^\beta|-1 = \sum_{k \in \mathbb{N}} k\beta(k)
 =
 	\sum_{k \in \mathbb{N}, k < |z^\beta| } k\beta(k)
 \end{equs}
 Therefore, once $|z^\beta|$ is fixed, the numbers of possible $k$ and $\beta(k)$ are bounded. For a forest of populated multi-indices, if $|\tilde{z}^{\tilde{\beta}}|$ is fixed, the number of possible partitions $|\tilde{z}^{\tilde{\beta}}| = \sum_{j=1}^{n_{\tilde{\beta}}}|z^{\beta_j}|$ is finite, where $n_{\tilde{\beta}} \in \mathbb{N}_{+}$ is the number of individual multi-indices in the forest $\tilde{z}^{\tilde{\beta}}$.
\end{proof}
\begin{proposition}\label{prop:well_posedness_Delta_2}
	The dual coproduct $\Delta_2$ is well-defined.
\end{proposition}
\begin{proof}
	We also use the form
	\begin{equs}
		 \Delta_2 \tilde{z}^{\tilde{\mu}} = \sum_{\tilde{z}^{\tilde{\beta}}, \tilde{z}^{\tilde{\alpha}} \in \CM} C_{\Delta_2}(\tilde{z}^{\tilde{\beta}}, \tilde{z}^{\tilde{\alpha}},  \tilde{z}^{\tilde{\mu}}) \tilde{z}^{\tilde{\beta}}  \otimes  \tilde{z}^{\tilde{\alpha}}.
	\end{equs}
	Then we need to show 
	\begin{itemize}
		\item $C_{\Delta_2}(\tilde{z}^{\tilde{\beta}}, \tilde{z}^{\tilde{\alpha}},  \tilde{z}^{\tilde{\mu}})$ is finite for any $\tilde{z}^{\tilde{\beta}}, \tilde{z}^{\tilde{\alpha}},\tilde{z}^{\tilde{\mu}} \in \CM$,
		\item the set $\{(\tilde{z}^{\tilde{\beta}} ,  \tilde{z}^{\tilde{\alpha}}): C_{\Delta_2}(\tilde{z}^{\tilde{\beta}}, \tilde{z}^{\tilde{\alpha}},  \tilde{z}^{\tilde{\mu}}) \ne 0\}$ is finite for any $\tilde{z}^{\tilde{\mu}} \in \CM$.
	\end{itemize}
	The first statement can be proved by the adjoint relation \eqref{adjoint_2} and the inner products \eqref{inner_2}. Through the definition of $\star_2$ we can see that $C_{\star_2}(\tilde{z}^{\tilde{\beta}}, \tilde{z}^{\tilde{\alpha}},  \tilde{z}^{\tilde{\mu}})$ is finite and accordingly $C_{\Delta_2}(\tilde{z}^{\tilde{\beta}}, \tilde{z}^{\tilde{\alpha}},  \tilde{z}^{\tilde{\mu}})$ is finite. The second statement is a corollary of the Lemma \ref{lemma:length_sum}. Since $|\tilde{z}^{\tilde{\mu}}|$ is finite, $|\tilde{z}^{\tilde{\beta}}|$ and $| \tilde{z}^{\tilde{\alpha}}|$ are finite. Therefore, due to Lemma~\ref{lemma:length_finite}, the forms of $\tilde{z}^{\tilde{\beta}}$ and $\tilde{z}^{\tilde{\alpha}}$ are finite if $C_{\Delta_2}(\tilde{z}^{\tilde{\beta}}, \tilde{z}^{\tilde{\alpha}},  \tilde{z}^{\tilde{\mu}}) \ne 0$.
\end{proof}

Then, we can present the composition law for multi-indice $B$-series in the following way

\begin{theorem} 
\label{thm:main_theorem_composition}
For linear maps $a$ and $b$ as defined in \eqref{character}  with $b(\emptyset)=1$, the composition of two multi-indices $B$-series satisfies
\begin{equs}
	B(a,h,f,\cdot)\circ B(b,h,f,y) = B(b\star_2a,h,f,y)
\end{equs}
where for $\tilde{z}^{\tilde{\beta}} \in \CM$
\begin{equs}
	(b\star_2 a) (\tilde{z}^{\tilde{\beta}}) := <b \otimes a, \Delta_2 \tilde{z}^{\tilde{\beta}}>,
\end{equs}	
	and the pairing is defined as 
\begin{equs}
	<b \otimes a, \Delta_2 \tilde{z}^{\tilde{\beta}}> = (b \otimes a)(\Delta_2 \tilde{z}^{\tilde{\beta}}).
\end{equs}
\end{theorem}

\begin{proof} One starts with
	\begin{equs}
		&B\left(a,h,f,B(b,h,f,y)\right) 
		\\ =& 
		a(\emptyset) B(b,h,f,y)   +  \sum_{z^\alpha \in\mathbf{M} } 	
		\frac{h^{|z^\alpha|} a(z^\alpha)}{S(z^\alpha)}F_f[z^\alpha](B(b,h,f,y)).
	\end{equs}
	For $z^\alpha \in\mathbf{M}$, Taylor expansion to $F_f[z^\alpha](B(b,h,f,y))$ around $y$ results in 
	\begin{equs}
		F_f[z^\alpha]\left(B(b,h,f,y)\right) 
		= \sum_{k \in \mathbb{N}}\frac{1}{k!}
		\partial^k
		F_f[z^\alpha](y)\left(B(b,h,f,y)-y\right)^k.
	\end{equs}
	Since $b(\emptyset)=1$, one has
	\begin{equs}
		&F_f[z^\alpha]\left(B(b,h,f,y)\right) 
		\\=&
		\sum_{k \in \mathbb{N}}\frac{1}{k!}
		\partial^kF_f[z^\alpha](y)
		\left(\sum_{z^\beta \in\mathbf{M} } 	
		\frac{h^{|z^\beta|} b(z^\beta)}{S(z^\beta)}F_f[z^\beta](y)\right)^k
		\\=& 
		F_f[z^\alpha](y)+
		\sum_{k \in \mathbb{N}_{+}}
		\sum_{z^{\beta_1},\ldots,z^{\beta_k} \in \mathbf{M}}
		\frac{1}{k!}\partial^kF_f[z^\alpha](y)\prod_{j=1}^k
		\left(\frac{b(z^{\beta_j})h^{|z^{\beta_j}|}}{S(z^{\beta_j})}F_f[z^{\beta_j}](y)\right).
	\end{equs}
	By Proposition ~\ref{prop:F_derivatives_multiindices},
	\begin{equs}
		&\lefteqn{F_f[z^\alpha]\left(B(b,h,f,y)\right) 
		- F_f[z^\alpha](y)} 
		\\=&
	   \sum_{k \in \mathbb{N}_{+}}
		\sum_{z^{\beta_1},\ldots,z^{\beta_k} \in \mathbf{M}}
		\frac{1}{k!}
		\frac{h^{\sum_{j=1}^k|z^{\beta_j}|}}{\prod_{j=1}^kS(z^{\beta_j})}\prod_{j=1}^kb(z^{\beta_j})	
		F_f\left[\tilde{\prod}_{j=1}^kz^{\beta_j}\star_2z^\alpha\right](y).
	\end{equs}
	Consider the permutation among $z^{\beta_1},\ldots,z^{\beta_k}$, and one has
	\begin{equs}
		&F_f[z^\alpha]\left(B(b,h,f,y)\right) -F_f[z^\alpha](y)
		\\=&
		\sum_{k \in \mathbb{N}_{+}}
		\sum_{\tilde{z}^{\tilde{\beta}} \in \CM_k}
		{k \choose r_1,\ldots,r_m}
		\frac{1}{k!}	\frac{h^{|\tilde{z}^{\tilde{\beta}} |}}{\prod_{j=1}^k S(z^{\beta_j})}b(\tilde{z}^{\tilde{\beta}} )F_f\left[\tilde{z}^{\tilde{\beta}} \star_2z^\alpha\right](y)
		\\=&
		\sum_{k \in \mathbb{N}_{+}}
		\sum_{\tilde{z}^{\tilde{\beta}}  \in \CM_k}
		\frac{h^{|\tilde{z}^{\tilde{\beta}} |}}{S(\tilde{z}^{\tilde{\beta}} )}b(\tilde{z}^{\tilde{\beta}} )F_f\left[\tilde{z}^{\tilde{\beta}} \star_2z^\alpha\right](y)			
	\end{equs}
	where $r_1,\ldots,r_m$ are the repetition among $z^{\beta_1},\ldots,z^{\beta_k}$. i.e.,
	\begin{equs}
		\tilde{\prod}_{j=1}^kz^{\beta_j} = \tilde{\prod}_{i=1}^m\left(z^{\gamma_i}\right)^{r_i}
	\end{equs}
	where $z^{\gamma_i} \ne z^{\gamma_j}$ if $i \ne j$.
	Then by the fact
	\begin{equs}
		\tilde{z}^{\tilde{\beta}} \star_2z^\alpha &= \sum_{z^\mu \in \mathbf{M}}
		\frac{<\tilde{z}^{\tilde{\beta}} \star_2z^\alpha, z^\mu>}{S(z^\mu)} z^\mu
		= \sum_{z^\mu \in \mathbf{M}}\frac{<\tilde{z}^{\tilde{\beta}}  \otimes z^\alpha, \Delta_2z^\mu>}{S(z^\mu)} z^\mu
	\end{equs}
	and $b(\emptyset) = 1$ we have
	\begin{equs}
		&B\left(a,h,f,B(b,h,f,y)\right)
		\\=& 
		a(\emptyset) 	
		\left(b(\emptyset)y+\sum_{z^\beta \in\mathbf{M}} 	
		\frac{h^{|z^\beta|} b(z^\beta)}{S(z^\beta)}F_f[z^\beta](y)\right)
		+\sum_{z^\alpha \in \mathbf{M}}
		\frac{a(z^\alpha)b(\emptyset)h^{|z^\alpha|}} {S(z^\alpha)}
		F_f[z^\alpha](y)
		\\+&
		\sum_{z^\mu \in \mathbf{M}}
		\sum_{z^\alpha \in \mathbf{M}}
		\sum_{k \in \mathbb{N}_{+}}
		\sum_{\tilde{z}^{\tilde{\beta}}  \in \CM_k}
		\frac{<\tilde{z}^{\tilde{\beta}}  \otimes z^\alpha, \Delta_2z^\mu>}
		{S(\tilde{z}^{\tilde{\beta}} )S(z^\alpha)}
		\frac{a(z^\alpha)b(\tilde{z}^{\tilde{\beta}} )h^{|z^\mu|}}{S(z^\mu)}F_f[z^\mu](y),
	\end{equs}
	where we used the Lemma~\ref{lemma:length_sum} that $h^{|z^\mu|} = h^{|\tilde{z}^{\tilde{\beta}} | + |z^\alpha|}$ if $<\tilde{z}^{\tilde{\beta}}  \otimes z^\alpha, \Delta_2z^\mu> \ne 0$. 
	Moreover, we have
	\begin{equs}
		&a(\emptyset) 	
		\left(b(\emptyset)y+\sum_{z^\beta \in\mathbf{M} } 	
		\frac{h^{|z^\beta|} b(z^\beta)}{S(z^\beta)}F_f[z^\beta](y)\right) 
		\\=& a(\emptyset)b(\emptyset)y + \sum_{z^\beta \in\mathbf{M} } 	a(\emptyset)\frac{h^{|z^\beta|} b(z^\beta)}{S(z^\beta)}F_f[z^\beta](y)
		\\=& a(\emptyset)b(\emptyset)y +  \sum_{z^\beta \in\mathbf{M} } \sum_{z^\mu \in\mathbf{M}} 
		\frac{<\emptyset  \otimes z^\beta, \Delta_2z^\mu>}
		{S(\emptyset)S(z^\beta)}
		\frac{a(\emptyset )b(z^\beta)h^{|z^\mu|}}{S(z^\mu)}F_f[z^\mu](y),
	\end{equs}
	and 
	\begin{equs}
		&\sum_{z^\alpha \in \mathbf{M}}
		\frac{a(z^\alpha)b(\emptyset)h^{|z^\alpha|}} {S(z^\alpha)}
		F_f[z^\alpha](y)
		\\=& \sum_{z^\alpha \in \mathbf{M}} \sum_{z^\mu \in\mathbf{M}} 
		\frac{<z^\alpha \otimes \emptyset, \Delta_2z^\mu>}
		{S(z^\alpha)S(\emptyset)}
		\frac{a(z^\alpha)b(\emptyset)h^{|z^\mu|}}{S(z^\mu)}F_f[z^\mu](y)
	\end{equs}
	where we used the fact that
	\begin{equs}
		\frac{<\emptyset  \otimes z^\beta, \Delta_2z^\mu>}
		{S(\emptyset)S(z^\beta)} = 1
	\end{equs}
	is the only non-zero coefficient and it is  obtained when $z^\mu = z^\beta$. It is the same for 
	\begin{equs}
		\frac{<z^\alpha \otimes \emptyset, \Delta_2z^\mu>}
		{S(z^\alpha)S(\emptyset)} =1 \quad \text{with } z^\mu = z^\alpha .
	\end{equs}
	Finally since $\frac{<\tilde{z}^{\tilde{\beta}}  \otimes z^\alpha, \Delta_2z^\mu>}
	{S(\tilde{z}^{\tilde{\beta}} )S(z^\alpha)}$ is the coefficient in front of the term $\tilde{z}^{\tilde{\beta}}  \otimes z^\alpha$ in $\Delta_2z^\mu$, we can conclude
	\begin{equs}
		B\left(a,h,f,B(b,h,f,y)\right) = B[b \star_2a,h,f,y].
	\end{equs}

\end{proof}

\section{Substitution of multi-indice $B$-series}
\label{sec::substitution}

Chartier, Hairer, and Vilmart \cite{CHV,MR2657947} showed that the substitution law for $B$-series with trees is efficient in interpreting modified integrator method which is an extension to the backward error analysis theory. Therefore, we would like to explore the substitution law for multi-indices B-series.
Firstly, the substitution is defined as: 
\begin{equs} 
\label{substitution_def}
	B(a,h,f,\cdot)	\circ_s \tilde f := 	B(a,h,\tilde f,y).
\end{equs}
It amounts to replacing the function $f$ in the multi-indice $B$-series by another function $\tilde f$. This can be used for backward error analysis when approximating $f$ by $\tilde f$. If we choose to  replace $f$ by $h^{-1}B(b,h,g,y)$, by applying the definition of elementary differentials, one has
\begin{equs}
	B(a,h,f,\cdot)	\circ_s h^{-1}B(b,h,g,y) = a(\emptyset)y+\sum_{z^\beta \in \mathbf{M}}
	\frac{a(z^\beta)h^{|z^\beta|}}{S(z^\beta)}
	\hat{F}_g[z^\beta](y),
\end{equs}
where
\begin{equs}
	\hat{F}_g[z^\beta](y) &=  \prod_{k \in \mathbb{N}}\left(h^{-1}\partial^kB(b,h,g,y)\right)^{\beta(k)}
	\\& = h^{ \tiny \text{$-\sum_{k \in \mathbb{N}}\beta(k)$}} \prod_{k \in \mathbb{N}}\left(\partial^kB(b,h,g,y)\right)^{\beta(k)}
	\\& = h^{\tiny \text{$-|z^\beta|$}} \prod_{k \in \mathbb{N}}\left(\partial^kB(b,h,g,y)\right)^{\beta(k)}.
\end{equs}
Similar to the substitution law for tree-based $B$-series \cite{CA}, we need an insertion product for multi-indices defined below.

\begin{definition}
\label{def:insertion_multi-indices}
The insertion of $z^\beta \in \mathbf{M}$ into $z^\alpha \in \mathbf{M}$ is defined to be
\begin{equs}
	z^\beta \blacktriangleright z^\alpha : =\sum_{k \in \mathbb{N}}\left(D^{k}z^{\beta}\right)
	\left(\partial_{z_{k}}z^\alpha\right).
\end{equs}	
\end{definition}

\noindent We note that $z^\beta \blacktriangleright z^\alpha$ is equivalent to a pre-Lie product defined in \cite[Def.~4.3]{Li23} up to coefficients $\frac{1}{k!}$. It is summing up all possible ways to replace one single $z_k$ in $z^\alpha$ by $z^{\beta}$. It can be seen from the directional derivative 
\begin{equs}
	\partial_{z_{k}}z^\alpha = \alpha(k)\frac{z^\alpha}{z_k}
\end{equs}
that $\alpha(k)$ indicates that we have $\alpha(k)$ choices of $z_k$ to be replaced and $\frac{1}{z_k}$ means that one $z_k$ is replaced.
Since $z_k$ has arity $k$, to ensure that $z^\beta \blacktriangleright z^\alpha$ is a populated multi-indice, we need $D^{k}$ to increase the arity of $z^{\beta}$ by $k$. Moreover, we can extend this insertion from $z^\alpha \in \mathbf{M}$ to $\tilde{z}^{\tilde{\alpha}} \in \CM_{+}$ by letting $\partial_{z_k}$ obey the Leibniz rule. i.e., for any $ \CM_{+}  \ni \tilde{z}^{\tilde{\beta}} = 	\tilde{\prod}_{j=1}^n z^{\beta_j}$
\begin{equs} \label{Leibniz_partial_k}
		\partial_{z_k}	\tilde{z}^{\tilde{\beta}} = \sum_{i=1}^n   \left(\tilde{\prod}_{j\ne i} z^{\beta_j}\right)\, \tilde{\bullet} \, \partial_{z_k} z^{\beta_i}.
\end{equs}
\begin{example}
	Suppose $z^\beta = z_0z_1$ and $z^\alpha = z_0^2z_1z_2$. One can easily check that $z^\beta$ and $z^\alpha$ are populated. Then,
	\begin{equs}
		z^\beta \blacktriangleright z^\alpha &= z_0z_12(z_0z_1z_2)+(z_0z_2+z_1^2)z_0^2z_2+(z_0z_3+3z_1z_2)(z_0^2z_1)
		\\&=
		2z_0^2z_1^2z_2+z_0^3z_2^2+z_0^2z_1^2z_2+z_0^3z_1z_3+3z_0^2z_1^2z_2,
	\end{equs}
	where the right-hand side is a linear combination of populated multi-indices.
\end{example}

\begin{remark}
It can be seen from
\begin{equs}
		F_f\left[	z^\beta \blacktriangleright z^\alpha\right]
		= \sum_{k \in \mathbb{N}} \alpha(k)F_f[ z^\alpha] \frac{F_f[D^kz^\beta]}{F_f[z_k]} 
		= \sum_{k \in \mathbb{N}} \alpha(k)F_f[ z^\alpha]
		\frac{\partial^kF_f[z^\beta]}{f^{(k)}} 
\end{equs}
that the insertion $\blacktriangleright$ for multi-indices is the counterparty of tree insertion $\widehat{\blacktriangleright}$ defined in \cite[Sec.~3.4]{BM22}, since
\begin{equs}
	F\left[\sigma \widehat{\blacktriangleright} \tau \right]
	&= 
	\sum_{v \in N_\tau}
	F\left[(P_v(\tau)\star_{\tiny{\text{GO}}}\sigma) \, \widehat{\triangleright}_vT_v(\tau)\right]
	\\&= 
	\sum_{v \in N_\tau}
	F[\tau]\frac{\partial^{k(v)}F[\sigma]}{f^{(k(v))}},
\end{equs}
where $\sigma$ and $\tau$ are trees, and $k(v)$ is the number of children of the node $v$, $ P_v(\tau ) $  corresponds to subtrees attached to $v$ in $\tau$, 
$ T_v(\tau) $ is the trunk obtained by removing all the branches attached to $v$ and $ \widehat{\triangleright}_v  $ identifies the root of the tree on the left hand side of this product with the node $v$ of the tree of the right hand side. 
\end{remark}	

Furthermore, the simultaneous insertion of $n$ multi-indices is defined as the following which is the Guin--Oudom construction on the insertion product $\blacktriangleright$.

\begin{definition} 
\label{def:star_1_def}
	The simultaneous insertion of $z^{\beta_j} \in \mathbf{M}$ into $z^\alpha \in \mathbf{M}$ is
\begin{equs}
	\tilde{\prod}_{j=1}^n z^{\beta_j} \star_1 z^\alpha:= \sum_{k_1,...,k_n\in \mathbb{N}}  
	\left(\prod_{j=1}^nD^{k_j}z^{\beta_j}\right)\left[\left(\prod_{j=1}^n\partial_{z_{k_j}}\right)z^\alpha\right] .
\end{equs}
	We also put an additional restriction that $n = |z^\alpha|$ in the definition of $\star_1$ as we would like to use it to describe the substitution law,  (see the proof of Theorem \ref{main_theorem_substitution} for reason).
\end{definition}
We can extend this simultaneous insertion from $z^\alpha \in \mathbf{M}$ to $\tilde{z}^{\tilde{\alpha}} \in \CM_{+}$ by Leibniz rule on $\left(\prod_{j=1}^n\partial_{z_{k_j}}\right)$ as in Eq. \eqref{Leibniz_partial_k}. The according modification of the additional constrain would be $n = |\tilde{z}^{\tilde{\alpha}}|$.

\begin{remark}
$\left(\prod_{j=1}^n\partial_{z_{k_j}}\right )z^\alpha$ encapsulates the feature that we can only insert up to $\alpha(k)$ multi-indices $z^{\beta_j}$ to replace $z_k$, which corresponds to the requirement on tree insertion that we can only insert up to one tree to replace a single node.
\end{remark}

Moreover, $\star_1$ commutes with $D$ which leads to 
\begin{lemma}
\label{lemma:commutative}
	For any $z^{\beta_j} \in \mathbf{M}$ and $\tilde{z}^{\tilde{\alpha}} \in \CM_{+}$
\begin{equs}
	\tilde{\prod}_{j=1}^n z^{\beta_j} \star_1 D^m\tilde{z}^{\tilde{\alpha}} = D^m\left(\tilde{\prod}_{j=1}^n z^{\beta_j} \star_1 \tilde{z}^{\tilde{\alpha}}\right).
\end{equs}
\end{lemma}
This statement can be verified simply through induction on $m$ and $n$, and a similar proof is provided in \cite[Lemma 4.2, 4.6]{Li23}.	

Then, we can define an operator on the linear map $b: \CM \mapsto \mathbb{R}$ as defined in \eqref{character} and $ \tilde{z}^{\tilde{\beta}} \in \CM_{+}$
\begin{equs}
	M_b(\tilde{z}^{\tilde{\beta}}) 
	:= \sum_{\tilde{z}^{\tilde{\alpha}} \in \CM_{|\tilde{z}^{\tilde{\beta}}|}} \frac{b(\tilde{z}^{\tilde{\alpha}})
	h^{|\tilde{z}^{\tilde{\alpha}}|}}{S(\tilde{z}^{\tilde{\alpha}})}
	\tilde{z}^{\tilde{\alpha}} \star_1 \tilde{z}^{\tilde{\beta}},
\end{equs}
which is a morphism of $\star_2$ and one has the following proposition.

\begin{proposition}	
\label{prop:morphism}	
For any $z^\beta \in \mathbf{M}$ and $\tilde{z}^{\tilde{\beta}} \in \CM_{+}$
\begin{equs}\label{eq:morephism_M}
	M_b\left(\tilde{z}^{\tilde{\beta}}\star_2z^\beta\right) 
	= M_b\left(\tilde{z}^{\tilde{\beta}}\right)\star_2M_b\left(z^\beta\right).
\end{equs}
\end{proposition}

\begin{proof}
Since $\CM_{+} = \cup_{n \in \mathbb{N}_{+}} \CM_n$, the proof is equivalent to show that $\tilde{z}^{\tilde{\beta}} \in \CM_{n_{\tilde{\beta}}}$ satisfies Equation \ref{eq:morephism_M}, for any $n_{\tilde{\beta}} \in \mathbb{N}_{+}$. Suppose that $\tilde{z}^{\tilde{\beta}} \in \CM_{n_{\tilde{\beta}}}$,  $\tilde{z}^{\tilde{\beta}} = \tilde{\prod}_{j=1}^{n_{\tilde{\beta}}} z^{\beta_j}$, and $n_{\tilde{\alpha}} = |\tilde{z}^{\tilde{\beta}}|+|z^\beta|$ , then by Lemma~\ref{lemma:length_sum}
	\begin{equs}
		M_b\left(\tilde{z}^{\tilde{\beta}}\star_2z^\beta\right)
		=&
		\sum_{\tilde{z}^{\tilde{\alpha}} \in \CM_{n_{\tilde{\alpha}}}} \frac{b(\tilde{z}^{\tilde{\alpha}})
		h^{|\tilde{z}^{\tilde{\alpha}}|}}{S(\tilde{z}^{\tilde{\alpha}})} \tilde{z}^{\tilde{\alpha}} 
		\star_1
		\left(\prod_{j=1}^{n_{\tilde{\beta}}} z^{\beta_j}D^{n_{\tilde{\beta}}}z^\beta\right)
		\\=&
		\frac{1}{{n_{\tilde{\alpha}}\choose r_1,\ldots, r_{m_{\tilde{\alpha}}}}}
		\sum_{z^{\alpha_1},\ldots,z^{\alpha_{n_{\tilde{\alpha}}}} \in \mathbf{M}}
		\sum_{k_1,\ldots,k_{n_{\tilde{\alpha}}}\in \mathbb{N}}
		\\& 
		\frac{\prod_{j=1}^{n_{\tilde{\alpha}}}b(z^{\alpha_j})h^{|z^{\alpha_j}|}}
		{\prod_i^{m_{\tilde{\alpha}}}r_i!\prod_{j=1}^{n_{\tilde{\alpha}}}S(z^{\alpha_j})}
		\left(\prod_{j=1}^{n_{\tilde{\alpha}}} D^{k_j}z^{\alpha_j}\right)
		\left[
		\left(\prod_{j=1}^{n_{\tilde{\alpha}}}\partial_{z_{k_j}}\right)
		\left(\prod_{j=1}^{n_{\tilde{\beta}}} z^{\beta_j}D^{n_{\tilde{\beta}}}z^\beta\right)\right]
		\\=&
		\frac{1}{n_{\tilde{\alpha}}!}
		\sum_{z^{\alpha_1},\ldots,z^{\alpha_{n_{\tilde{\alpha}}}} \in \mathbf{M}}
		\sum_{k_1,\ldots,k_{n_{\tilde{\alpha}}}\in \mathbb{N}}
		\frac{\prod_{j=1}^{n_{\tilde{\alpha}}}b(z^{\alpha_j})h^{|z^{\alpha_j}|}}
		{\prod_{j=1}^{n_{\tilde{\alpha}}}S(z^{\alpha_j})}
		\\& 
		\left(\prod_{j=1}^{n_{\tilde{\alpha}}} D^{k_j}z^{\alpha_j}\right)
		\left[
		\left(\prod_{j=1}^{n_{\tilde{\alpha}}}\partial_{z_{k_j}}\right)
		\left(\prod_{j=1}^{n_{\tilde{\beta}}} z^{\beta_j}D^{n_{\tilde{\beta}}}z^\beta\right) \right],
	\end{equs}
where we used the multiplicativity of $b$, r is the repetition among $z^{\alpha_j}$, and $m_{\tilde{\alpha}}$ is the number of distinct forms of multi-indices among $z^{\alpha_j}$. Then, by the Leibniz rule and Lemma ~\ref{lemma:commutative}, 
\begin{equs}
	& \sum_{k_1,\ldots,k_{n_{\tilde{\alpha}}}\in \mathbb{N}}
	\left(\prod_{j=1}^{n_{\tilde{\alpha}}} D^{k_j}z^{\alpha_j}\right)
	\left[
	\left(\prod_{j=1}^{n_{\tilde{\alpha}}}\partial_{z_{k_j}}\right)
\left(\prod_{j=1}^{n_{\tilde{\beta}}} z^{\beta_j}D^{n_{\tilde{\beta}}}z^\beta\right) \right]
	\\=&
	\sum_{l \le n_{\tilde{\alpha}}}
	\sum_{\{\tilde{z}^{\tilde{\alpha}^{(1)}},\tilde{z}^{\tilde{\alpha}^{(2)}}\} \in \CP_l(\tilde{z}^{\tilde{\alpha}})}
	\sum_{k_1,\ldots,k_{n_{\tilde{\alpha}}}\in \mathbb{N}}
	{n_{\tilde{\alpha}} \choose l}
	\left[\prod_{i=1}^l D^{k^{(1)}_i}z^{\alpha^{(1)}_i}
	\left(\prod_{i=1}^l\partial_{z_{k^{(1)}_i}}\right)\prod_{j=1}^{n_{\tilde{\beta}}} z^{\beta_j}\right]
	\\&
	\left[\prod_{i=1}^{n_{\tilde{\alpha}}-l}D^{k^{(2)}_i}z^{\alpha^{(2)}_i}
	\left(\prod_{i=1}^{n_{\tilde{\alpha}}-l}\partial_{z_{k^{(2)}_i}}\right)D^{n_{\tilde{\beta}}}z^\beta\right]
	\\=&
	\sum_{\{\tilde{z}^{\tilde{\alpha}^{(1)}},\tilde{z}^{\tilde{\alpha}^{(2)}}\} \in \CP_{|\tilde{z}^{\tilde{\beta}}|}(\tilde{z}^{\tilde{\alpha}})}
	{n_{\tilde{\alpha}} \choose |\tilde{z}^{\tilde{\beta}}|}
	\left(\tilde{z}^{\tilde{\alpha}^{(1)}} \star_1 \tilde{z}^{\tilde{\beta}}\right) 
	D^{n_{\tilde{\beta}}}\left(\tilde{z}^{\tilde{\alpha}^{(2)}} \star_1 z^\beta\right)
\end{equs}
where $\CP_l(\tilde{z}^{\tilde{\alpha}})$ is the set of all possible partitions of $\tilde{z}^{\tilde{\alpha}}$ into two parts
	\begin{equs}
		\tilde{z}^{\tilde{\alpha}^{(1)}} \tilde{\bullet}\tilde{z}^{\tilde{\alpha}^{(2)}} 
		= \tilde{z}^{\tilde{\alpha}} \quad \text{ with } \tilde{z}^{\tilde{\alpha}^{(1)}} \in \CM_l,	
		\tilde{z}^{\tilde{\alpha}^{(2)}} \in \CM_{n_{\tilde{\alpha}}-l},
	\end{equs}
	\begin{equs}
		\tilde{z}^{\tilde{\alpha}^{(1)}}  = \tilde{\prod}_{i=1}^l z^{\alpha^{(1)}_i}, \quad 
		\tilde{z}^{\tilde{\alpha}^{(2)}}  = \tilde{\prod}_{i=1}^{n_{\tilde{\alpha}}-l} z^{\alpha^{(2)}_i},
	\end{equs}
	and $k^{(j)}_i$ is the $k$ corresponding to $z^{\alpha^{(j)}_i}$ for  $j=1,2$.
	In the last equality we used the fact that only terms with $l = |\tilde{z}^{\tilde{\beta}}|$ is non-zero. The reason is that $n_{\tilde{\alpha}} = |\tilde{z}^{\tilde{\beta}}|+|z^\beta|$. If $l < |\tilde{z}^{\tilde{\beta}}|$, then $n_{\tilde{\alpha}}-l > |z^\beta|$ will cause $\left(\prod_{i=1}^{n_{\tilde{\alpha}}-l}\partial_{z_{k^{(2)}_i}}\right)D^{n_{\tilde{\beta}}}z^\beta = 0$. It is the same for $l > |\tilde{z}^{\tilde{\beta}}|$.
	Since $\tilde{z}^{\tilde{\alpha}^{(1)}} \star_1 \tilde{z}^{\tilde{\beta}} \in \CM_{n_{\tilde{\beta}}}$, we have
	\begin{equs}
		\left(\tilde{z}^{\tilde{\alpha}^{(1)}} \star_1 \tilde{z}^{\tilde{\beta}}\right) D^{n_{\tilde{\beta}}}
		\left(\tilde{z}^{\tilde{\alpha}^{(2)}} \star_1 z^\beta\right) 
		= \left(\tilde{z}^{\tilde{\alpha}^{(1)}} \star_1 \tilde{z}^{\tilde{\beta}}\right) 
		\star_2
		\left(\tilde{z}^{\tilde{\alpha}^{(2)}} \star_1 z^\beta\right).
	\end{equs}		
	We then consider the permutation
	\begin{equs}
		&\sum_{z^{\alpha_1},\ldots,z^{\alpha_{n_{\tilde{\alpha}}}} \in \mathbf{M}}
		\sum_{\{\tilde{z}^{\tilde{\alpha}^{(1)}},\tilde{z}^{\tilde{\alpha}^{(2)}}\} \in \CP_{|\tilde{z}^{\tilde{\beta}}|}(\tilde{z}^{\tilde{\alpha}})}
		\\=&
		{|\tilde{z}^{\tilde{\beta}}| \choose r^{(1)}_1,\ldots,r^{(1)}_{m_1}}
		{n_{\tilde{\alpha}}-|\tilde{z}^{\tilde{\beta}}| \choose r^{(2)}_1,\ldots,r^{(2)}_{m_2}}
		\sum_{\tilde{z}^{\tilde{\alpha}^{(1)}} \in \CM_{|\tilde{z}^{\tilde{\beta}}|}}
		\sum_{\tilde{z}^{\tilde{\alpha}^{(2)}} \in \CM_{n_{\tilde{\alpha}}-|\tilde{z}^{\tilde{\beta}}|}} 	 	
	\end{equs}
where $r^{(1)}$ is the repetition among individual multi-indices in the forest  $\tilde{z}^{\tilde{\alpha}^{(1)}}$, and $r^{(2)}$ is for $\tilde{z}^{\tilde{\alpha}^{(2)}}$.		
	Finally, we obtain
	\begin{equs}
		M_b\left(\tilde{z}^{\tilde{\beta}}\star_2z^\beta\right)
		=&
		\sum_{\tilde{z}^{\tilde{\alpha}^{(1)}} \in \CM_{|\tilde{z}^{\tilde{\beta}}|}}
		\sum_{\tilde{z}^{\tilde{\alpha}^{(2)}} \in \CM_{n_{\tilde{\alpha}}-|\tilde{z}^{\tilde{\beta}}|}}
		\frac{b(\tilde{z}^{\tilde{\alpha}^{(1)}})h^{|\tilde{z}^{\tilde{\alpha}^{(1)}}|}}{S(\tilde{z}^{\tilde{\alpha}^{(1)}})} 
		\\&
		\frac{b(\tilde{z}^{\tilde{\alpha}^{(2)}})h^{|\tilde{z}^{\tilde{\alpha}^{(2)}}|}}{S(\tilde{z}^{\tilde{\alpha}^{(2)}})}
		\left(\tilde{z}^{\tilde{\alpha}^{(1)}} 
		\star_1 
		\tilde{z}^{\tilde{\beta}}\right) 
		\star_2
		\left( \tilde{z}^{\tilde{\alpha}^{(2)}} \star_1 z^\beta\right)
		\\=&
		\left(\sum_{\tilde{z}^{\tilde{\alpha}^{(1)}} \in \CM_{|\tilde{z}^{\tilde{\beta}}|}}
		\frac{b(\tilde{z}^{\tilde{\alpha}^{(1)}})h^{|\tilde{z}^{\tilde{\alpha}^{(1)}}|}}{S(\tilde{z}^{\tilde{\alpha}^{(1)}})}
		\tilde{z}^{\tilde{\alpha}^{(1)}}
		\star_1 
		\tilde{z}^{\tilde{\beta}}\right) 
		\star_2
		\\&
		\left(\sum_{\tilde{z}^{\tilde{\alpha}^{(2)}} \in \CM_{|z^\beta|}}
		\frac{b(\tilde{z}^{\tilde{\alpha}^{(2)}})h^{|\tilde{z}^{\tilde{\alpha}^{(2)}}|}}{S(\tilde{z}^{\tilde{\alpha}^{(2)}})}
		\tilde{z}^{\tilde{\alpha}^{(2)}}
		\star_1 z^\beta\right)
		\\=&
		M_b\left(\tilde{z}^{\tilde{\beta}}\right)\star_2M_b\left(z^\beta\right).
	\end{equs}
\end{proof}

\begin{remark}
The morphism Proposition \eqref{prop:morphism} can also be proved by the universality of \cite[Eq. 4.11]{Li23} which can be interpreted as the morphism written as the level of the pre-Lie products. It is similar in the spirit to \cite[Thm 3.25]{BM22} where such a statement has been established for decorated trees.
	\end{remark}

Suppose that the coproduct $\Delta_1$ is the dual of $\star_1$, which means they have the adjoint relation
\begin{equs}
	<\tilde{z}^{\tilde{\beta}} \otimes \tilde{z}^{\tilde{\alpha}} , \Delta_1 \tilde{z}^{\tilde{\mu}} > 
	= <\tilde{z}^{\tilde{\beta}}  \star_1 \tilde{z}^{\tilde{\alpha}} , \tilde{z}^{\tilde{\mu}} >
\end{equs}
for $\tilde{z}^{\tilde{\beta}}, \tilde{z}^{\tilde{\alpha}} , \tilde{z}^{\tilde{\mu}} \in \CM_{+}$. We need the following lemma in showing the well-posedness of $\Delta_1$.
\begin{lemma}\label{lemma:length_equal}
	For any  $\tilde{z}^{\tilde{\beta}}, \tilde{z}^{\tilde{\alpha}} , \tilde{z}^{\tilde{\mu}} \in \CM_{+}$, if $<\tilde{z}^{\tilde{\beta}} \otimes \tilde{z}^{\tilde{\alpha}} , \Delta_1 \tilde{z}^{\tilde{\mu}} > \ne 0$, the following holds
	\begin{equs}
		|\tilde{z}^{\tilde{\mu}}| = |\tilde{z}^{\tilde{\beta}}|.
	\end{equs}
\end{lemma}
\begin{proof}
	$<\tilde{z}^{\tilde{\beta}} \otimes \tilde{z}^{\tilde{\alpha}} , \Delta_1 \tilde{z}^{\tilde{\mu}} > \ne 0$ indicates that $\tilde{z}^{\tilde{\mu}}$ can be obtained by ``inserting" $\tilde{z}^{\tilde{\beta}}$ to $\tilde{z}^{\tilde{\alpha}}$ using the product $\star_1$. Through the definition of $\star_1$ we can see that 
	$D$ will not change the length (it only replace some $z_k$ by $z_{k+1}$) but $\partial_{z_{k}}$ will reduce it by $1$ and $n_{\tilde\beta}$ is the number of individual populated multi-indices in the forest $\tilde{z}^{\tilde{\beta}}$. Therefore, the length will be shorten exactly by $n_{\tilde\beta}$. 
		\begin{equs}
		|\tilde{z}^{\tilde{\mu}}| = |\tilde{z}^{\tilde{\beta}}| + |\tilde{z}^{\tilde{\alpha}}| - n_{\tilde\beta}.
	\end{equs}
	Since we have the constrain in the definition of $\star_1$ that $|\tilde{z}^{\tilde{\beta}}| = n_{\tilde\alpha}$, it follows that 
	\begin{equs}
		|\tilde{z}^{\tilde{\mu}}| = |\tilde{z}^{\tilde{\beta}}|.
	\end{equs}
\end{proof}
\begin{proposition}
	The dual coproduct $\Delta_1$ is well-defined.
\end{proposition}
\begin{proof}
	One can show it in the same way as the proof of the well-posedness of $\Delta_2$ (Proposition~\ref{prop:well_posedness_Delta_2}). The only difference is to replace Lemma~\ref{lemma:length_sum} by Lemma~\ref{lemma:length_equal}
\end{proof}
Then we have the substitution law for multi-indices $B$-series as below.	
\begin{theorem}	
\label{main_theorem_substitution}
If $b(\emptyset) = 0$ and $(b\star_1 a) (\emptyset)$ is set to be $a(\emptyset)$, one has
\begin{equs}
	B(a,h,f,y)	\circ_s  h^{-1}B(b,h,g,y) = B(b \star_1 a, h,g,y)
\end{equs}
where for $z^\gamma \in \mathbf{M}$
\begin{equs}
	(b \star_1 a) (z^\gamma) = <b \otimes a, \Delta_1 z^\gamma>,
\end{equs}
	and the pairing is defined as 
	\begin{equs}
		<b \otimes a, \Delta_1 z^\gamma> = (b \otimes a)(\Delta_1 z^\gamma).
	\end{equs}
\end{theorem}

\begin{proof}
We start with the right-hand side
	\begin{equs}
		B(b\star_1 a, h,g,y) =
		a(\emptyset)y+
		\sum_{z^{\bar{\beta}} \in \mathbf{M}} 
		\frac{<b \otimes a, \Delta_1 z^{\bar{\beta}}>h^{|z^{\bar{\beta}}|}}
		{S(z^{\bar{\beta}})} 
		F_g[z^{\bar{\beta}}](y).
	\end{equs}
	For each $z^{\bar\beta} \in \mathbf{M}$ we have
	\begin{equs}
		&<b \otimes a, \Delta_1 \bar{\beta}>h^{|z^{\bar{\beta}}|}
		\\=&
		\sum_{\tilde{z}^{\tilde{\beta}} \in \CM}
		\sum_{z^\beta \in \mathbf{M}}
		\frac{<\tilde{z}^{\tilde{\beta}}\otimes z^\beta, \Delta_1 z^{\bar{\beta}}>}{S(\tilde{z}^{\tilde{\beta}})
		S(z^\beta)}b(\tilde{z}^{\tilde{\beta}})a(z^\beta)h^{|z^{\bar{\beta}}|}.
	\end{equs}
	Notice that as long as $<\tilde{z}^{\tilde{\beta}}\otimes z^\beta, \Delta_1 z^{\bar{\beta}}>\ne 0$, by Lemma~\ref{lemma:length_equal} we have
	\begin{equs}
		h^{|z^{\bar{\beta}}|} = h^{|\tilde{z}^{\tilde{\beta}}|}.
	\end{equs}
Then by the adjoint relationship,
	\begin{equs}
		&<b \otimes a, \Delta_1 \bar{\beta}>h^{|z^{\bar{\beta}}|}
		\\=&
		\sum_{\tilde{z}^{\tilde{\beta}} \in \CM}
		\sum_{z^\beta \in \mathbf{M}}
		\frac{<\tilde{z}^{\tilde{\beta}}\star_1 z^\beta,  z^{\bar{\beta}}>}{S(\tilde{z}^{\tilde{\beta}})
		S(z^\beta)}b(\tilde{z}^{\tilde{\beta}})a(z^\beta)h^{|\tilde{z}^{\tilde{\beta}}|}.
	\end{equs}
	Therefore,
	\begin{equs}
		&B(b\star_1 a, h,g,y)-a(\emptyset)y
		\\=&
		\sum_{z^\beta \in \mathbf{M}}
		\sum_{\tilde{z}^{\tilde{\beta}} \in \CM_{|z^\beta|}}
		\frac{b(\tilde{z}^{\tilde{\beta}})a(z^\beta)h^{|\tilde{z}^{\tilde{\beta}}|}
		}{S(\tilde{z}^{\tilde{\beta}})S(z^\beta)}
		F_g\left[		
		\sum_{z^{\bar{\beta}} \in \mathbf{M}}
		\frac{<\tilde{z}^{\tilde{\beta}}\star_1 z^\beta,  z^{\bar{\beta}}>}{S(z^{\bar{\beta}})}
		z^{\bar{\beta}}
		\right](y)
		\\=&
		\sum_{z^\beta \in \mathbf{M}}
		\sum_{\tilde{z}^{\tilde{\beta}} \in \CM_{|z^\beta|}}
		\frac{b(\tilde{z}^{\tilde{\beta}})a(z^\beta)h^{|\tilde{z}^{\tilde{\beta}}|}
		}{S(\tilde{z}^{\tilde{\beta}})S(z^\beta)}
		F_g\left[\tilde{z}^{\tilde{\beta}}\star_1 z^\beta\right](y)
		\\=&
		\sum_{z^\beta \in \mathbf{M}}
		\frac{a(z^\beta)}{S(z^\beta)}
		F_g\left[M_b(z^\beta)\right](y).
	\end{equs}
	Then, the proof boils down to show that
	\begin{equs}
		F_g\left[M_b(z^\beta)\right] = h^{|z^\beta|}\hat{F}_g[z^\beta] \quad \text{for any } z^\beta \in \mathbf{M}. 
	\end{equs}
	We prove this by induction on $|z^\beta|$. Firstly, 
	\begin{equs}
		F_g\left[M_b(z_0)\right] 
		&= F_g\left[\sum_{\tilde{z}^{\tilde{\alpha}} \in \CM_{1}} \frac{b(\tilde{z}^{\tilde{\alpha}})
		h^{|\tilde{z}^{\tilde{\alpha}}|}}{S(\tilde{z}^{\tilde{\alpha}})} \tilde{z}^{\tilde{\alpha}} \star_1 z_0\right]
		\\&=F_g\left[\sum_{z^\alpha \in \mathbf{M}} \frac{b(z^{\alpha})h^{|z^{\alpha}|}}{S(z^{\alpha})} 
		z^{\alpha} \star_1 z_0\right]
		\\&=B(b,h,g,y) =h^{|z_0|} \hat{F}_g[z_0],		
	\end{equs}
	where we used the facts that $b(\emptyset) = 0$ and $z^\alpha \star_1 z_0 = z^\alpha$. From Corollary~\ref{corolla_decomposition}, 
	one can see that any $z^\beta \in \mathbf{M} \backslash\{z_0\}$ can be expressed as $\tilde{\prod}_{j=1}^nz^{\beta_j} \star_2 z_0$ for some $n \in \mathbb{N}_{+}$.
	Thus, the induction hypothesis is
	\begin{equs}
		F_g\left[M_b(z^{\beta_j})\right] =  h^{|z^{\beta_j}|}
			\hat{F}_g[z^{\beta_j}] \quad \text{for } j=1,\ldots,n. 
	\end{equs}
	Then, by Proposition \ref{prop:F_derivatives_multiindices} and \ref{prop:morphism},
	\begin{equs}
		F_g\left[M_b(z^{\beta})\right] &= F_g\left[M_b\left(\tilde{\prod}_{j=1}^nz^{\beta_j} \star_2 z_0\right)\right] 
		\\&=
		F_g\left[M_b\left(\tilde{\prod}_{j=1}^nz^{\beta_j}\right) \star_2 M_b(z_0)\right] 
		\\&=
		F_g\left[\tilde{\prod}_{j=1}^nM_b(z^{\beta_j}) \star_2 M_b(z_0)\right] 
		\\&=
		\prod_{j=1}^nF_g\left[M_b(z^{\beta_j})\right] \partial^n F_g\left[M_b(z_0)\right].
	\end{equs}
	Finally, by the induction base and hypothesis,
			\begin{equs}
			F_g\left[M_b(z^{\beta})\right] &=
			\prod_{j=1}^n\prod_{k \in \mathbb{N}} \left(\partial^kB(b,h,g,y)\right)^{\beta_j(k)}\partial^n
			B(b,h,g,y) 
			\\&=  h^{|z^\beta|} \hat{F}_g[z^\beta].
		\end{equs}
\end{proof}

\section{Connection with local affine equivariant methods}
\label{sec::affine_equivariant_methods}	
In the context of multi-indices, one considers  $ \mathfrak X = \mathcal{C}^{\infty}(\mathbb{R}, \mathbb{R})$. A numerical method $ \varphi $  is then a map that belongs to $\mathcal{C}^{\infty}(\mathfrak X, \mathfrak X)$.
	
\begin{definition} 
\label{def_affine_equivariant} 
A numerical method $ \varphi $ is said  to be \textit{local} if for every $ \sigma \in \mathfrak X $
\begin{equs}
	\text{supp} ( \varphi(\sigma) ) \subset \text{supp}(\sigma).
\end{equs}
It is \textit{affine equivariant} if for every $ a, b \in \mathbb{R}^* $ and $f \in \mathfrak X$, one has
\begin{equs}
	\varphi(a f + b)  =a \varphi(f) + b.
\end{equs}
\end{definition}

It has been shown the following nice characterisation on $B$-series in \cite[Cor.  8.4]{MV16}

\begin{theorem} 
\label{local_affine_equivariant}
If a map from $\mathcal{C}^{\infty}(\mathfrak X, \mathfrak X)$ is local and affine equivariant, then its Taylor development is a $B$-series.
\end{theorem}
In fact, this statement is a special result of a general statement for  $  \mathcal{C}^{\infty}(\mathbb{R}^d, \mathbb{R}^d)$ with $d \geq  1$ where now, one has to replace Butcher series by aromatic Butcher series (see \cite[Thm.~8.3]{MV16}). 
Aromatic Butcher series are an extension of Butcher series that contain trees with a loop at the root that represents divergences of vector fields in the elementary differentials. For instance,  one cannot encode the elementary differential $f^jf^k\partial_{ij}f^i$ with $i \ne k$ through rooted trees as, in a rooted tree $\Forest{[i[i][j]]}$, $\partial_{ij}$ is associated with $f^if^j$. Here, $\partial_{ij}$ is a shorthand notation for the coordinate derivatives $\partial_{x_i}\partial_{x_j}$. The element $f^jf^k$ is associated with $\partial_{jk}$ in $\Forest{[i[j][k]]}$. However, we can represent the desired elementary differential through the aromatic tree
	\begin{equs}
	F_f [\begin{tikzpicture}
		\draw[line width=0.3mm] (0,0) ellipse (0.2 and 0.15);  
		\fill (0.05,0.4) circle(0.07) node[right, font = \tiny] {j}; 
		\fill (0.05,0.15) circle(0.07)  node[right, font = \tiny] {i}; 
		\draw[line width=0.35mm] (0.05,0.15) -- (0.05,0.4);
		\fill (0.5,0) circle(0.07) node[ right, font = \tiny] {k}; 
	\end{tikzpicture}] = f^jf^k\partial_{ij}f^i.
\end{equs}
In \cite{MMMV16}, the authors have used a stronger assumption on the numerical method for getting a characterisation of Butcher series.  Nevertheless in the one-dimensional case, as there is no partial derivative, we can show that the elementary differentials encoded by aromatic trees are factorised by elementary differentials of multi-indice $B$-series to characterise Butcher series.

In \cite[Def.~5.3]{MV16}, they define the notion of composition map $\kappa  :  \mathbb{N} \rightarrow \mathbb{N}$ with finite support
\begin{equs}
	\text{card} \left\lbrace  j \in  \mathbb{N} | \kappa(j)  \neq 0 \right\rbrace < + \infty. 
\end{equs}
They also define its size by $ |\kappa| $ by
\begin{equs}
	|\kappa| : = \sum_{j \in \mathbb{N}} 
	\kappa(j).
\end{equs}
One obtains easily the following correspondence

\begin{proposition} 
\label{composition_maps}
Composition maps defined in \cite[Def.~5.3]{MV16} with finite support are exactly the multi-indices introduced in \cite{OSSW,LOT}.
\end{proposition}
According to \cite[Sec.~2.6]{MV16}, an aromatic forest is a (directed, finite) graph with at most one outgoing edge from each node. An aromatic tree is an aromatic forest with $n$ nodes and $n -1$ arrows. 
For any aromatic forest, one can associate  a composition map $\kappa$ with finite support by setting $ \kappa(j) $ to be the number of nodes having $j$ incoming edges. 
 
 \begin{proposition} 
 \label{composition_map_populated}
 The composition map $ \kappa $ associated to an aromatic tree is a populated multi-indice.
 \end{proposition}
 
\begin{proof}
An aromatic tree has  $n$ nodes and $n -1$ arrows  which means its associated multi-indice satisfies the population condition \eqref{populated_1} . One has  
\begin{equation*}
 	[\kappa] = \sum_{j \in \mathbb{N}} (1-k) \kappa(j)
 	= \sum_{j \in \mathbb{N}}  \kappa(j)
 	 - \sum_{j \in \mathbb{N}} j \kappa(j)
 	 = n - (n-1) = 1,
\end{equation*}
where have used the fact that $ \sum_{j \in \mathbb{N}}  \kappa(j) $ corresponds to the number of nodes and $ \sum_{j \in \mathbb{N}} j \kappa(j) $ to the number of arrows.
\end{proof}
 \begin{remark}
 	After modifying the condition ``{\it we cannot  choose $ z_{l_1}, \cdots,  z_{l_k}$ to be all $ z_0 $ as long as the number of variables in $z^{\beta'}$ is larger than the number of nodes to be attached.} " in the Algorithm \ref{algorithm:build_tree} or allowing $z_0$ to be the root for $z^\beta \ne z_0$ in the first step of the algorithm, one can build aromatic trees. For example, let us revisit Example \ref{ex:build_tree}. If we allow to use $z_0$ to decorate all incoming edges of the root before using up all the variables left in $z^{\beta'}$ (i.e., finally $z_1$ is left not attached to the rooted tree), we obtain
 	
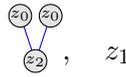
\begin{equs}
 		\begin{tikzpicture}[scale=0.2,baseline=-5]
 			\coordinate (root) at (0,-1);
 			\coordinate (right) at (1,2);
 			\coordinate (left) at (-1,2);
 			\draw[symbols] (root) -- (right);
 			\draw[symbols] (root) -- (left);
 			\node[var] (rootnode) at (left) {\tiny{$ z_0 $}};
 			\node[var] (rootnode) at (right) {\tiny{$ z_0 $}};
 			\node[var] (rootnode) at (root) {\tiny{$  z_{2} $}};
 		\end{tikzpicture}, \quad  z_1
 	\end{equs}
 	which is exactly the aromatic tree
 	
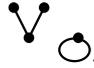
\begin{equs}
 		\begin{tikzpicture}
 			\fill (0,0) circle(0.07); 
 			\fill (-0.2,0.4) circle(0.07); 
 			\fill (0.2,0.4) circle(0.07); 
 			\draw[line width=0.3mm] (0,0) --  (-0.2,0.4);
 			\draw[line width=0.3mm] (0,0) --  (0.2,0.4);
 			\draw[line width=0.3mm] (0.6,-0.16) ellipse (0.2 and 0.15);  
 			\fill (0.65,0) circle(0.07); 
 		\end{tikzpicture} .
 	\end{equs}
 \end{remark}

By fixing a populated multi-indice $\kappa$, we denote by $ \Gamma_{\kappa} $ the aromatic trees whose composition map is equal to $ \kappa $. It has been noticed in \cite[Sec.~7.6]{MV16}, for $ \tau \in \Gamma_{\kappa} $ and one-dimensional $f$
\begin{equs}
	F_f[\tau] = \prod_{j \in \mathbb{N}}
	(f^{(j)})^{\kappa(j)} = F_f[z^{\kappa}].
\end{equs}
For example,
	\begin{equs}
		F_f [\begin{tikzpicture}
			\draw[line width=0.3mm] (0,0) ellipse (0.2 and 0.15);  
			\fill (0.05,0.4) circle(0.07); 
			\fill (0.05,0.15) circle(0.07); 
			\draw[line width=0.35mm] (0.05,0.15) -- (0.05,0.4);
			\fill (0.5,0) circle(0.07); 
		\end{tikzpicture}] = F_f[\Forest{[[][]]}  ] = f^2f^{(2)} = F_f[z_0^2 z_2],
	\end{equs}
		\begin{equs}
		F_f [\begin{tikzpicture}
			\fill (0,0) circle(0.07); 
			\fill (-0.2,0.4) circle(0.07); 
			\fill (0.2,0.4) circle(0.07); 
			\draw[line width=0.3mm] (0,0) --  (-0.2,0.4);
			\draw[line width=0.3mm] (0,0) --  (0.2,0.4);
			\draw[line width=0.3mm] (0.45,-0.16) ellipse (0.2 and 0.15);  
			\fill (0.5,0) circle(0.07); 
		\end{tikzpicture}] 
		=
		F_f[
		\begin{forest}
			[[[]][]]
		\end{forest}
		]
		=
		F_f[
		\begin{forest}
			[[[][]]]
		\end{forest}
		]
		= f^2f^{(1)}f^{(2)}
		= F_f[z_0^2 z_1z_2]
		.
	\end{equs}
The image of $ F_f $ has always dimension one, i.e., $\text{dim}(F_f \left[\left\langle  \Gamma_{\kappa}\right\rangle\right] ) = 1 $ where $\left\langle  \Gamma_{\kappa}\right\rangle$ is the free vector space of aromatic trees. In other words, the elementary differentials of the aromatic trees of a given composition $ \kappa $ all collapse to the same elementary differential. We can  make Theorem~\ref{local_affine_equivariant} and therefore the statement in \cite{MV16} more precise via Theorem~\ref{affine_equivariant_multi_indices_intro}. We recall to the reader its statement below
\begin{quote}
{\it{If a map from $\mathcal{C}^{\infty}(\mathfrak X, \mathfrak X)$ is local and affine equivariant, then its Taylor development is a multi-indices $B$-series. The choice of the multi-indice $B$-series is unique}.}
\end{quote}
\begin{proof}[of Theorem \ref{affine_equivariant_multi_indices_intro}] \label{Proof_Theorem_1_5} The only thing that we need to check is the uniqueness of the multi-indice B-series.
As the elementary differentials encoded by aromatic trees are factorised by those encoded by multi-indices,	it boils down to prove that multi-indices elementary differentials are linear independent. This means that one has
	\begin{equs}
\forall f \in \mathfrak X, \; \;	 \sum_{z^\beta \in \mathbf{M}} \lambda_{z^\beta} F_f[z^\beta] = 0 \Leftrightarrow \lambda_{z^\beta} =   0
	 \end{equs}
	where the sum runs over all populated multi-indices $ z^{\beta} $,  and $ \lambda_{\cdot} $ has a finite support.
We can therefore choose a specific $f$ for discriminating the various $ z^{\beta} $.
Let $ m $ be the highest integer such that  there exists $ z^{\beta} $ with $  \lambda_{z^\beta} \neq 0 $ and $ \beta(m) \neq 0$. We define
\begin{equs}
f(x) :=	f_{t_0,...,t_m}(x) = \sum_{k=0}^m t_k \frac{x^k}{k!}
\end{equs} 
where $ t_0 = 1$ since for a populated multi-indice if $\beta(k)$ is fixed for every $k \ne 0$ the $\beta(0)$ is also fixed. We have chosen a function $f$ parametrised by $t_1,...,t_m$. Then, one has
\begin{equs}
	F_f[z^\beta](0) =  \prod_{k=0}^m
	(t_k)^{\beta(k)}
\end{equs}
One can observe that we obtain a monomial in $  t_1,...,t_m$ which is uniquely associated to a multi-indice $z^\beta$.
This family of monomials is clearly a free basis which allows us to conclude. 
\end{proof}


\end{document}